\documentclass[12pt]{article}

\usepackage{amssymb}

\usepackage{amsmath}

\usepackage{amsthm}

\usepackage{mathrsfs}

\usepackage{url}

\newcommand{\N}{\mathbb{N}}
\newcommand{\Z}{\mathbb{Z}}
\newcommand{\Q}{\mathbb{Q}}
\newcommand{\R}{\mathbb{R}}
\newcommand{\C}{\mathbb{C}}

\renewcommand{\b}{\mathfrak{b}}
\newcommand{\bbar}{\overline{\mathfrak{b}}}
\newcommand{\g}{\mathfrak{g}}
\newcommand{\h}{\mathfrak{h}}
\newcommand{\n}{\mathfrak{n}}
\newcommand{\nbar}{\overline{\mathfrak{n}}}

\newcommand{\Nbar}{\overline{N}}
\newcommand{\Bbar}{\overline{B}}

\newcommand{\tensor}{\otimes}

\newcommand{\comment}[1]{}

\DeclareMathOperator{\Hom}{Hom}
\DeclareMathOperator{\Ind}{Ind}
\DeclareMathOperator{\Lie}{Lie}

\newtheorem{theorem}{Theorem}
\newtheorem{proposition}[theorem]{Proposition}
\newtheorem{corollary}[theorem]{Corollary}
\newtheorem{lemma}[theorem]{Lemma}
\newtheorem{definition}[theorem]{Definition}

\newtheorem*{remark}{Remark}

\newtheorem*{maintheorem1}{Theorem~\ref{laindEScor}}
\newtheorem*{maintheorem2}{Theorem~\ref{LoefflerES}}

\newcommand{\verma}[2]{M_{#1}(#2)}

\newcommand{\eps}{\varepsilon}
\newcommand{\env}{\mathcal{U}}

\newcommand{\cato}{\mathcal{O}}
\newcommand{\catobar}{\overline{\mathcal{O}}}

\newcommand{\an}{\text{an}}
\newcommand{\cs}{\text{c}}
\newcommand{\cl}{\text{cl}}
\newcommand{\la}{\text{la}}
\newcommand{\lp}{\text{lp}}
\newcommand{\pol}{\text{pol}}

\newcommand{\sm}{\text{sm}}

\DeclareMathOperator{\ad}{ad}
\DeclareMathOperator{\anInd}{an-Ind}
\DeclareMathOperator{\im}{im}
\DeclareMathOperator{\smInd}{sm-Ind}
\DeclareMathOperator{\supp}{Supp}
\DeclareMathOperator{\Mod}{Mod}

\begin{document}

% \begin{frontmatter}



\title{An analogue of the BGG resolution for locally analytic principal series}


\author{Owen T. R. Jones}

\maketitle

\section*{Abstract}
Let $\mathbf{G}$ be a connected reductive quasisplit algebraic group over a field $L$ which is a finite extension of the $p$-adic numbers. We construct an exact sequence modelled on (the dual of) the BGG resolution involving locally analytic principal series representations for $\mathbf{G}(L)$. This leads to an exact sequence involving spaces of overconvergent $p$-adic automorphic forms for certain groups compact modulo centre at infinity.

% \linenumbers


\section{Introduction} \label{sintro}

Locally analytic representation theory is the study of a certain class of representations of $L$-analytic groups over $K$, where $L$ is a finite extension of $\Q_p$ and $K$ is a spherically complete extension of $L$. It was systematically developed by Schneider and Teitelbaum in papers such as \cite{schnteit2}, \cite{schnteit}, \cite{schnteit4} and \cite{schnteit5}. It plays an important role in the $p$-adic local Langlands correspondence for $GL_2 (\Q_p)$.

It also has applications to overconvergent $p$-adic automorphic forms. For connected reductive groups which are compact modulo centre at infinity, spaces of overconvergent $p$-adic automorphic forms are defined by Loeffler in \cite{Loeffler} in terms of functions from a certain set to a locally analytic principal series representation for an Iwahori subgroup.

Let $G$ be the group of $L$-points of a connected reductive linear quasisplit algebraic group defined over $L$, $B$ and $\Bbar$ opposite Borels in $G$, $G_1$ a subgroup of $G$ admitting an Iwahori factorisation, such as an Iwahori subgroup, and $\Bbar_1 = \Bbar \cap G_1$. In this paper we study maps between locally analytic principal series $\Ind_{\Bbar_1}^{G_1} (\mu)$ for $G_1$, where $\Ind$ denotes locally analytic induction over $K$, which we assume is complete with respect to a  discrete valuation. Our approach is to exploit an isomorphism between $\Ind_{\Bbar}^{G} (\mu)(N)$, the subspace of functions in $\Ind_{\Bbar}^{G} (\mu)$ with support in $\Bbar N$, where $N$ is the unipotent radical of $B$, and $\mathcal{C}^{\la}_{\cs} (N, K_{\mu})$, the space of locally analytic functions $N \to K_{\mu}$ with compact support. Using the dense subspace $\mathcal{C}^{\lp}_{\cs} (N, K_{\mu}) \subseteq \mathcal{C}^{\la}_{\cs} (N, K_{\mu})$, which we call the space of locally polynomial functions, we prove the following theorem.

\begin{maintheorem1}
We have an exact sequence of $M$-representations
\begin{multline*}
0 \to V \tensor_{K} \smInd_{\Bbar_1}^{G_1} (\mathbf{1}) \to \Ind_{\Bbar_1}^{G_1} (\lambda) \to \bigoplus_{w \in W^{(1)}} \Ind_{\Bbar_1}^{G_1} (w \cdot \lambda) \to \\
\cdots \to \bigoplus_{w \in W^{(i)}} \Ind_{\Bbar_1}^{G_1} (w \cdot \lambda) \to \cdots \to \Ind_{\Bbar_1}^{G_1} (w_0 \cdot \lambda) \to 0.
\end{multline*}
coming from the BGG resolution for $V^{*}$.
\end{maintheorem1}

Here $V$ is the irreducible finite-dimensional algebraic representation of $G$ with highest weight $\lambda$, $\smInd$ is smooth induction, $\mathbf{1}$ is the trivial character, $W^{(i)}$ denotes the elements of the Weyl group of length $i$, $w_0$ is the longest element of the Weyl group and $w \cdot \lambda = w (\lambda + \rho) - \rho$ where $\rho$ is half the sum of the positive roots.

By taking $G_1$ to be an Iwahori subgroup, we can use Theorem~\ref{laindEScor} to construct the analogous exact sequence for locally analytic principal series $\Ind_{\Bbar}^{G} (\mu)$ for $G$, which has been established by quite different methods in \cite{OrlikStrauch2}. Another consequence of Theorem~\ref{laindEScor} is the following exact sequence between spaces $M(e, K_{\mu})$ of overconvergent $p$-adic automorphic forms of weight $\mu$ for a group compact modulo centre at infinity whose group of $L$-points is $G$.

\begin{maintheorem2}
If $\lambda \in X(\mathbf{T})$ is a dominant integral arithmetical character then we have an exact sequence
\begin{multline*}
0 \to M(e, K_{\lambda})^{\text{cl}} \to M(e, K_{\lambda}) \to \bigoplus_{w \in W^{(1)}} M(e, K_{w \cdot \lambda}) \to \\
\cdots \to \bigoplus_{w \in W^{(i)}} M(e, K_{w \cdot \lambda}) \to \cdots \to M(e, K_{w_0 \cdot \lambda}) \to 0
\end{multline*}
\end{maintheorem2}

Here $M(e, K_{\lambda})^{\text{cl}}$ denotes the so-called classical subspace. After the first inclusion, the maps in this exact sequence are constructed from maps of the form $\theta_{\alpha, w \cdot \lambda}^{\text{aut}}: M(e, K_{w \cdot \lambda}) \to M(e, K_{s_{\alpha} w \cdot \lambda})$, where $w \in W$ and $\alpha \in \Phi^{+}$ satisfy $l(s_{\alpha} w) = l(w) + 1$. These are the analogue of the maps $\theta^{k - 1}$ from \cite{Colemanclassover} between the spaces of overconvergent $p$-adic modular forms of weight $2 - k$  and $k$. In \cite{Colemanclassover}, Coleman used $\theta^{k - 1}$ to prove a sufficient condition for classicality in terms of small slope. Using Theorem~\ref{LoefflerES} we can establish a necessary and sufficient condition for belonging to the classical subspace.

\subsection{The structure of the paper}

In \S\ref{ssubspaces} we define certain subspaces of locally analytic principal series in which we will be interested. In \S\ref{sliealg} we establish results about representations of a split semisimple Lie algebra, including the exactness of a certain duality functor. In \S\ref{sdefmaps} we use maps between Verma modules to construct maps between particular subspaces of locally analytic principal series. We use the BGG resolution to construct a sequence of $\env(\g)$-modules involving these subspaces, for $G$ semisimple in \S\ref{sBGG} and for $G$ reductive in \S\ref{sredgps}. We then show that the first three terms of this sequence are exact in \S\ref{sfirstthreeterms}. In \S\ref{slocan} we build on this to prove Theorem~\ref{laindEScor}, and deduce the exactness of the original sequence. We prove the analogue of Theorem~\ref{laindEScor} involving locally analytic principal series for $G$ rather than a subgroup with an Iwahori factorisation in \S\ref{sG}. In \S\ref{san} we turn our attention to analytic principal series and prove that the first three terms of the exact sequence in Theorem~\ref{laindEScor} remain exact when we restrict to analytic principal series.

Finally, we give some applications of our results to overconvergent $p$-adic automorphic forms for groups compact modulo centre at infinity. In \S\ref{sappi} we briefly sketch the definition of overconvergent $p$-adic automorphic forms given by Chenevier in \cite{Chenevierfern} and prove a three-term exact sequence involving certain spaces of overconvergent automorphic forms. This material is contained in \cite{Chenevierfern}, citing an earlier version of this paper, but is included here for completeness. In \S\ref{sappii} we briefly outline the definition of overconvergent $p$-adic automorphic forms given by Loeffler in \cite{Loeffler} and prove Theorem~\ref{LoefflerES}.

\subsection{Notation} \label{ssnotation}

Fix a prime $p$. Let $L$ be a finite extension of $\Q_p$ and  let $K$ be an extension of $L$ which is complete with respect to a discrete valuation (by Lemma~1.6 in \cite{schnfunctionalanalysis} this means it is spherically complete).

Let $\mathbf{G}$ be a connected reductive linear algebraic group defined over $L$ which is quasi-split over $L$ and split over $K$. Choose a Borel subgroup $\mathbf{B}$ which is defined over $L$. Write $\mathbf{N}$ for its unipotent radical (which is defined over $L$). Choose a maximal $L$-split torus in $\mathbf{B}$ and let $\mathbf{T}$ be its centraliser in $\mathbf{G}$. Then $\mathbf{T}$ is a Levi factor in $\mathbf{B}$ and a maximal torus in $\mathbf{G}$ which is defined over $L$. It is not necessarily split over $L$, but by assumption it splits over $K$. Let $\mathbf{\Bbar}$ denote the opposite Borel to $\mathbf{B}$ containing $\mathbf{T}$ and $\mathbf{\Nbar}$ its unipotent radical.

We write $G$ for $\mathbf{G}(L)$. We use bold letters to denote algebraic subgroups of $\mathbf{G}$. For any algebraic subgroup $\mathbf{J}$ of $\mathbf{G}$ defined over $L$ we write $J$ to denote $\mathbf{J}(L)$ and the lower case gothic letter $\mathfrak{j}$ to represent the corresponding Lie subalgebra of $\g = \Lie(G)$. The sole exception is that we will denote the Lie algebra of $T$ by $\h$, which is the standard notation in Lie algebra representation theory. Given a Lie algebra $\mathfrak{a}$ we write $\env(\mathfrak{a})$ for the universal enveloping algebra of $\mathfrak{a}$. Representations of $\mathfrak{a}$ are equivalent to $\env(\mathfrak{a})$-modules, and we use the two terms interchangeably. We write $S : \env(\mathfrak{a}) \to \env(\mathfrak{a})$ for the principal anti-automorphism of $\env(\mathfrak{a})$, given on monomials by $X_1 \cdots X_r \mapsto (-1)^r X_r \cdots X_1$. This is the unique algebra anti-automorphism of $\env(\mathfrak{a})$ extending $\mathfrak{a} \to \mathfrak{a}$, $X \mapsto -X$.

Let $\Phi$ denote the set of all roots of $\mathbf{G}$, $\Phi^{+} \subseteq \Phi$ the set of positive roots determined by our choice of $\mathbf{B}$ and $\Delta \subseteq \Phi^{+}$ the corresponding set of simple roots. Set $\Phi^{-} = \{ -\alpha: \alpha \in \Phi^{+} \}$. For each $\alpha \in \Phi^{+}$ let $H_{\alpha} \in \h$ denote its coroot and fix a non-zero element $E_{\alpha} \in \g_{\alpha}$. This determines a unique $F_{\alpha} \in \g_{-\alpha}$ such that $[E_{\alpha}, F_{\alpha}] = H_{\alpha}$.

Let $X(\mathbf{T}) = \Hom (\mathbf{T}, \mathbb{G}_m)$, with the group law written additively. Let $\h^{*}$ be the space of Lie algebra homomorphisms from $\h$ to $K$. Any $\mu \in X(\mathbf{T})$ gives an element in $\h^{*}$ by evaluating at $K$-points, restricting to $T$ and then differentiating at the origin. We denote this element again by $\mu$. We are only interested in elements of $\h^{*}$ coming from $X(\mathbf{T})$.

Let $W$ denote the Weyl group of $\mathbf{G}$, $W^{(i)}$ the subset of elements of length $i$ under the Bruhat ordering given by our choice of positive roots and $w_0$ the longest element of $W$. Let $\rho$ be half the sum of the positive roots. There is a natural action of $W$ on $X(\mathbf{T})$, and we define the affine action of $w \in W$ on $\mu \in X(\mathbf{T})$ by $w \cdot \mu = w(\mu + \rho) - \rho$. We define the affine action of $W$ on $\h^{*}$ similarly.

Let $\lambda \in X(\mathbf{T})$ be a dominant weight and let $\sigma: \mathbf{G} \to \mathbf{GL}_{s}$ be the irreducible finite-dimensional representation of $\mathbf{G}$ with highest weight $\lambda$ . As $\mathbf{T}$ is split over $K$, $\sigma$ is defined over $K$. Let $V$ denote $K^{s}$ with the action of $G$ given by $\sigma: G \to \mathbf{GL}_{s} (K)$.

Suppose $X$ is a paracompact locally $L$-analytic manifold and $U$ a $K$-vector space. We write $\mathcal{C}^{\la} (X, U)$ for the space of locally $L$-analytic functions from $X$ to $K$, and $\mathcal{C}^{\sm} (X, U)$ for the subspace of all smooth (i.e. locally constant) functions. The subspaces of compactly supported functions are denoted $\mathcal{C}^{\la}_{\cs} (X, U)$ and $\mathcal{C}^{\sm}_{\cs} (X, U)$ respectively. If $\Omega$ is an open and closed subset of $X$ then we write $\mathbf{1}_{\Omega} \in \mathcal{C}^{\sm} (X, K)$ for the indicator function of $\Omega$. For any $f \in \mathcal{C}^{\la} (X, U)$ we write $f_{| \Omega}$ for $f \mathbf{1}_{\Omega}$. If $Y$ is an open subset of $X$, $\mathbb{Y}$ is a rigid analytic space defined over $L$ and $\varphi: Y \to \mathbb{Y}(L)$ is locally analytic isomorphism which is compatible with all the charts of $X$ then we write $\mathcal{C}^{\an}(Y, K)$ for the subspace of $\mathcal{C}^{\la} (Y, K)$ consisting of functions $f: Y \to K$ such that $f \circ \varphi^{-1}$ comes from a holomorphic function on $\mathbb{Y}(L)$. We say $f \in \mathcal{C}^{\la} (X, K)$ is analytic on $Y$ if $f|_{Y} \in \mathcal{C}^{\an}(Y, K)$.

\section{Subspaces of $\Ind_{\Bbar}^{G}(\nu)$} \label{ssubspaces}

Now we recall some definitions and propositions from \cite{emertonjacquet2}, Emerton's forthcoming paper on the relation of his Jacquet module functor to parabolic induction, which contains a longer exposition of all the material in this section. For simplicity we often give definitions only in the cases we need them, rather than the more general versions found in \cite{emertonjacquet2}. All representations will be vector spaces over $K$, even if this is not explicitly mentioned.

\begin{definition}
Let $U$ be a barrelled, Hausdorff, locally convex $K$-vector space with an action of a locally $L$-analytic group $J$ by continuous $K$-linear automorphisms. We say $U$ is a \textbf{locally analytic representation} of $J$ if for every $u \in U$ the orbit map $J \to U, j \mapsto j u$ is in $\mathcal{C}^{\la}(J, U)$.
\end{definition}

If $U$ is a locally analytic representation of $H$ then we can differentiate the action of $J$ to get an action of $\mathfrak{j}$, or equivalently of its enveloping algebra $\env (\mathfrak{j})$, as explained in Remark 2.5 in \cite{emertonsummary}.

If $\nu$ is a locally analytic character $T \to GL_{1}(K)$ then let $K_{\nu}$ denote the corresponding one-dimensional representation of $T$. As every $\overline{b} \in \Bbar$ can be written uniquely as $\overline{n} t$ with $\overline{n} \in \Nbar$ and $t \in T$, we define the locally analytic parabolic induction to be
\[
\Ind_{\Bbar}^{G} (\nu) = \{ f \in \mathcal{C}^{\la}(G, K): f(\overline{n} t g) = \nu(t) f(g) \text{ for all } \overline{n} \in \Nbar, t \in T, g \in G \}
\]
with the right regular action of $G$: $g' f (g) = f (g g')$. This is a locally analytic representation of $G$, as explained in Proposition~2.1.1 of \cite{emertonjacquet2}.

The support of any $f \in \Ind_{\Bbar}^{G} (\nu)$ is an open and closed subset of $G$ which is invariant under multiplication on the left by $\Bbar$. Its image in $\Bbar \backslash G$ is therefore open and compact. We refer to this as the support of $f$, $\supp f$. If $\Omega$ is any open subset of $\Bbar \backslash G$ we let $\Ind_{\Bbar}^{G} (\nu) (\Omega)$ denote the subspace of elements whose support is contained in $\Omega$.

Since $N \cap \Bbar = \{ e \}$, the natural map $N \to \Bbar \backslash G$ given by $n \mapsto \Bbar n$ is an open immersion. We use this map to regard $N$ as an open subset of $\Bbar \backslash G$. By Lemma~2.3.6 of \cite{emertonjacquet2}, this open immersion induces a topological isomorphism
\begin{equation} \label{suppN}
\mathcal{C}^{\la}_{\cs}(N, K_{\nu}) \overset{\sim}{\longrightarrow} \Ind_{\Bbar}^{G} (\nu)(N).
\end{equation}

We extend the right translation action of $N$ on $\mathcal{C}^{\la}_{\cs} (N, K_{\nu})$ to a locally analytic action of $B$ by letting $t \in T$ act on $f \in \mathcal{C}^{\la}_{\cs} (N, K_{\nu})$ as follows:
\[
t f(n) = \nu(t) f (t^{-1} n t).
\]
On the other hand the action of $B$ on $\Ind_{\Bbar}^{G} (\nu)$ preserves $\Ind_{\Bbar}^{G} (\nu)(N)$ as $\Bbar N B = \Bbar N$, so we have an action of $B$ on $\Ind_{\Bbar}^{G} (\nu)(N)$. These actions make \eqref{suppN} $B$-equivariant.

As $\Ind_{\Bbar}^{G} (\nu)$ is a locally analytic representation of $G$ it also has an action of $\g$. If $X \in \g$ and $f \in \Ind_{\Bbar}^{G} (\nu)$ is $0$ on some open neighbourhood of $g \in G$ then $X f$ is also $0$ on this neighbourhood. Hence the action of $\g$ preserves $\Ind_{\Bbar}^{G} (\nu)(N)$, whence we get an action of $\g$ on  $\Ind_{\Bbar}^{G} (\nu)(N)$. Restricting it to $\b$ gives the same action as differentiating the $B$-action, so $\Ind_{\Bbar}^{G} (\nu)(N)$ is a $(\g, B)$-representation. We use \eqref{suppN} to transfer this action of $\g$ to $\mathcal{C}^{\la}_{\cs}(N, K_{\nu})$.

We now identify various subspaces of $\mathcal{C}^{\la}_{\cs} (N, K_{\nu})$, which by \eqref{suppN} correspond to subspaces of $\Ind_{\Bbar}^{G} (\nu)(N)$.

Define $\mathcal{C}^{\pol}(N, K)$, the ring of algebraic $K$-valued functions on $N$, to be the set of all functions $N \to K$ which come from global sections of the structure sheaf of $\mathbf{N}$ over $K$. We give $\mathcal{C}^{\pol}(N, K)$ its finest locally convex topology, so the natural injection into $\mathcal{C}^{\la}(N, K)$ is continuous. We let $N$ act by the right regular representation. We extend this to an action of $B$ by $t f (n) = f(t^{-1} n t)$. This makes $\mathcal{C}^{\pol}(N, K)$ an algebraic representation, in the sense that we may write it as a union of an increasing series of finite dimensional $B$-invariant subspaces, on each of which $B$ acts through an algebraic representation of $\mathbf{B}$. Each of these representations is a fortiori a locally analytic representation of $B$, so we can differentiate them to get actions of $\b$. These all agree, so we get an action of $\b$ on $\mathcal{C}^{\pol}(N, K)$. In fact because we have given $\mathcal{C}^{\pol}(N, K)$ its finest locally convex topology the action of $B$ makes it a locally analytic representation, as explained after Lemma~2.5.3 in \cite{emertonjacquet2}, which gives us another way of constructing this action of $\b$.

We let $B$ act on $K_{\nu}$ via $B \to B / N \cong T$. We define the space of polynomial functions on $N$ with coefficients in $K_{\nu}$ to be $\mathcal{C}^{\pol} (N, K_{\nu}) = \mathcal{C}^{\pol} (N, K) \tensor_{K} K_{\nu}$, equipped with the inductive, or equivalently projective, tensor product topology (cf. \S17 of \cite{schnfunctionalanalysis}). This has an action of $B$ by taking the tensor product of the actions on the two factors. Since both of these actions are locally analytic the action on $\mathcal{C}^{\pol} (N, K_{\nu})$ is locally analytic, so we have an action of $\b$. We now explain how to extend this to an action of $\g$.

We write $\n^k$ for $\{ E_1 \cdots E_k : E_{i} \in \n$ for all $1 \leq i \leq k \} \subseteq \env(\n)$. For any $\env(\n)$-module $M$ we define $M^{\n^{\infty}}$ to be the subspace of all $x \in M$ such that $\n^k x = \{ 0 \}$ for some positive integer $k$. We let $A_{\nu}$ denote the one-dimensional $K$-representation of $\Bbar$ on which $T$ acts as $\nu$ and $\Bbar$ acts via $\Bbar \to \Bbar / \Nbar \cong T$. From the discussion following Lemma~2.5.3 in \cite{emertonjacquet2} the map
\begin{align*}
\mathcal{C}^{\pol}(N, K_{\nu}) \overset{\sim}{\longrightarrow} \Hom_K (\env (\n), A_{\nu})^{\n^{\infty}} &&
f \longmapsto (u \mapsto (uf) (e))
\end{align*}
is an isomorphism of $\env (\n)$-modules, where the action of $\n$ on $\Hom_K (\env (\n), A_{\nu})$ is given by $X \phi (u) = \phi (u X)$ for all $X \in \n$, $\phi \in \Hom_K (\env (\n), K_{\nu})$ and $u \in \env (\n)$. By the Poincar\'e-Birkhoff-Witt theorem there is an isomorphism of $\env(\n)$-modules $\Hom_K (\env (\n), A_{\nu})^{\n^{\infty}} \overset{\sim}{\longrightarrow} \Hom_{\env(\bbar)} (\env (\g), A_{\nu})^{\n^{\infty}}$ where $\env(\g)$ is a $\env(\bbar)$-module by the multiplication on the left and has an action of $\n$ by multiplication on the right. Combining these we get an isomorphism of $\env(\n)$-modules
\[
\mathcal{C}^{\pol} (N, A_{\nu}) \overset{\sim}{\longrightarrow} \Hom_{\env(\bbar)} (\env (\g), A_{\nu})^{\n^{\infty}}
\]
We give $\Hom_{\env(\bbar)} (\env (\g), A_{\nu})^{\n^{\infty}}$ an action of $\g$ by $X \phi (u) = \phi (u X)$ for all $X \in \g$, $\phi \in \Hom_{\env(\bbar)} (\env (\g), A_{\nu})^{\n^{\infty}}$ and $u \in \env(\g)$. This map is then $\b$-equivariant. We use it to extend the action of $\b$ on $\mathcal{C}^{\pol} (N, K_{\nu})$ to an action of $\g$. By Lemma~2.5.8 in \cite{emertonjacquet2} this action is continuous and $\mathcal{C}^{\pol}(N, K_{\nu})$ is a $(\g, B)$-representation.

We define
\[
\mathcal{C}^{\lp}_{\cs}(N, K_{\nu}) = \mathcal{C}^{\pol}(N, K_{\nu}) \tensor_{K} \mathcal{C}^{\sm}_{\cs}(N, K)
\]
with the inductive tensor product topology (this coincides with the projective tensor product topology in this case, by Proposition~1.1.31 of \cite{emertonlocanvect}), where ``lp" is short for ``locally polynomial". We let $\g$ and $B$ act on both factors, where $\mathcal{C}^{\sm}_{\cs} (N, K)$ has the trivial action of $\g$. By Lemma~2.5.22 in \cite{emertonjacquet2} this action of $B$ is locally analytic. These actions make $\mathcal{C}^{\lp}_{\cs} (N, K_{\nu})$ a continuous $(\g, B)$-representation, i.e. the maps $\g \times \mathcal{C}^{\lp}_{\cs} (N, K_{\nu}) \to \mathcal{C}^{\lp}_{\cs} (N, K_{\nu})$ and $B \times \mathcal{C}^{\lp}_{\cs} (N, K_{\nu}) \to \mathcal{C}^{\lp}_{\cs} (N, K_{\nu})$ are both continuous.

Multiplication of algebraic functions by smooth functions gives a map
\[
\mathcal{C}^{\lp}_{\cs} (N, K_{\nu}) \longrightarrow \mathcal{C}^{\la}_{\cs} (N, K_{\nu})
\]
which is a continuous, $(\g, B)$-equivariant injection by Lemma~2.5.24 in \cite{emertonjacquet2}. We can think of its image as those locally analytic functions from $N$ to $K_{\nu}$ which are locally given by polynomials.

Now suppose that $X$ is an open subset of $N$, $\mathbb{X}$ is a rigid analytic affinoid ball defined over $L$ and $\varphi: X \to \mathbb{X}(L)$ is a locally analytic isomorphism compatible with all charts of $N$. Then the image of the map
\begin{align*}
\mathcal{C}^{\pol} (N, K_{\nu}) \to \mathcal{C}^{\an} (X, K_{\nu}) && f \to f|_{X}
\end{align*}
is dense in $\mathcal{C}^{\an} (X, K_{\nu})$. It follows that the image of $\mathcal{C}^{\lp}_{\cs} (N, K_{\nu})$ in $\mathcal{C}^{\la}_{\cs} (N, K_{\nu})$ is also dense.

\section{Lie algebra representations} \label{sliealg}

For this section only we change the notation. We work with a semisimple Lie algebra $\g$ over $K$ (instead of $L$) with a split Cartan subalgebra $\h$ (i.e. for all $X \in \h$ the eigenvalues of $\ad X$ are in $K$) and Borel subalgebra $\b \supseteq \h$ with nilpotent radical $\n$. Write $\bbar$ for the opposite Borel subalgebra to $\b$.

Let $\env(\g)$-$\Mod$ denote the category of $\env(\g)$-modules with morphisms given by linear maps which commute with the $\env(\g)$-actions. Let $\mathcal{C}$ denote the full subcategory of $\env(\g)$-$\Mod$ given by those modules $M$ on which $\h$ acts diagonalisably and the weight spaces are finite dimensional. It is easily checked that this is an abelian category.

Given $M \in \mathcal{C}$, the action of $\env(\g)$ on the dual representation $M^{*}$ is given by $u \phi (m) = \phi (S(u) m)$ for $\phi \in M^{*}, u \in \env(\g)$ and $m \in M$. If $\lambda$ is a weight of $M$ then we have an injection $i: (M_{\lambda})^{*} \to M^{*}$ by extending $\phi \in (M_{\lambda})^{*}$ by $0$ on all the other weight spaces $M_{\mu}$, and  $i( (M_{\lambda})^{*}) \subseteq  (M^{*})_{-\lambda}$. Moreover, for any $\phi \in (M^{*})_{-\lambda}$, $\phi (M_{\mu}) = 0$ unless $\lambda = \mu$, so we have equality: $i((M_{\lambda})^{*}) = (M^{*})_{-\lambda}$.

Since $M = \bigoplus_{\lambda \in \h^{*}} M_{\lambda}$ we get that $M^{*} \cong \prod_{\lambda \in \h^{*}} (M_{\lambda})^{*} \cong \prod_{\lambda \in \h^{*}} (M^{*})_{\lambda}$. We define $M^{\vee}$ to be
\[
M^{\vee} = \bigoplus_{\lambda \in \h^{*}} (M^{*})_{\lambda} \subseteq M^{*}
\]
Note that the action of $\h$ preserves $(M^{*})_{\lambda}$ and if $\phi \in (M^{*})_{\lambda}$ and $X \in \g_{\alpha}$ then $X \phi \in (M^{*})_{\lambda - \alpha}$. It follows that $M^{\vee}$ is a $\env(\g)$-submodule of $M^{*}$.

Clearly $\h$ acts diagonalisably on $M^{\vee}$. Since $(M^{*})_{\lambda} \cong (M_{-\lambda})^{*}$ the weight spaces are finite dimensional. Hence $M^{\vee} \in \mathcal{C}$.

\begin{lemma} \label{exact}
The contravariant functor $F: \mathcal{C} \to \mathcal{C}$ given by $M \to M^{\vee}$ is an anti-equivalence of categories. In particular it is exact.
\end{lemma}

\begin{proof}
First $F$ is a functor as any $\phi: M \to M'$ restricts to a map $M_{\lambda} \to M'_{\lambda}$ for any $\lambda \in \h^{*}$, and it is contravariant as $M \to M^{*}$ is. Now consider $(M^{\vee})^{\vee} \subseteq M^{**}$. It is a direct sum of its weight spaces and the $\lambda$ weight space is $((M^{\vee})_{- \lambda})^{*}$, which is $(M_{\lambda})^{**}$,  precisely the image of $M_{\lambda}$ under the canonical embedding $M \to M^{**}$. Thus $M^{\vee \vee}$ is isomorphic to $M$.

By a standard result in category theory $F$ is an anti-equivalence of categories if and only if $F$ is fully faithful and essentially surjective. This follows from the fact that $F^{2} (M) \cong M$ for any $M \in \mathcal{C}$.
\end{proof}

\begin{definition}
Let $M$ be a $\env(\g)$-module and let $\mathfrak{a}$ be a Lie subalgebra of $\g$. We say $\mathfrak{a}$ acts \textbf{locally finitely} on $x \in M$ if $x$ is contained in some finite dimensional $\env(\mathfrak{a})$-submodule of $M$. We define the \textbf{$\mathfrak{a}$-finite part of $M$} to be the subset of all elements of $M$ on which $\mathfrak{a}$ acts locally finitely. This is a $\env(\mathfrak{a})$-submodule of $M$.
\end{definition}

\begin{definition}
We define the \textbf{category~$\catobar$} to be the full subcategory of $\env(\g)$-$\Mod$ consisting of finitely generated modules on which $\nbar$ acts locally finitely and $\h$ acts diagonalisably.
\end{definition}

Note that $\catobar$ is closed under finite direct sums, submodules and quotients, and it contain all finite dimensional $\env(\g)$-modules. It is more usual to work with the category $\cato$, defined as in the above definition but with $\n$ instead of $\nbar$. For more background see \cite{humphreyscato}.

\begin{lemma} \label{subcat}
The category $\catobar$ is a subcategory of $\mathcal{C}$.
\end{lemma}

\begin{proof}
Suppose $M \in \catobar$. We know $\h$ acts diagonalisably on $M$ so we just have to show that the weight spaces are finite dimensional.

Let $X \subseteq M$ be a finite set which generates $M$ as a $\env(\g)$-module. As $\h$ acts diagonalisably on $M$ we have $M = \bigoplus_{\lambda \in \h^{*}} M_{\lambda}$, so we can write each $x \in X$ as a sum of finitely many elements of weight spaces. Replacing each $x \in X$ with the elements thus obtained, we may assume each $x \in X$ is a weight vector. As $\nbar$ acts locally finitely on $M$, $\env(\nbar) x$ is finite dimensional for each $x \in X$. It is also a $\env(\h)$-module, so we may pick a basis of weight vectors for it. Replacing each $x \in X$ by this basis we have a finite set $X$ of weight vectors in $M$ which generates $M$ as a $\env(\n)$-module.

Choose an ordering $\{ \alpha_1, \cdots, \alpha_r \}$ for $\Phi^{+}$. Then $\{ E_{\alpha_1}, \cdots, E_{\alpha_r} \}$ is a basis for $\n$, so by the Poincar\'e-Birkhoff-Witt theorem $\{ E_{\alpha_1}^{n_1} \cdots E_{\alpha_r}^{n_r} : n_i \geq 0 \ \forall i \}$ is a basis for $\env(\n)$. It follows that the set $\{ E_{\alpha_1}^{n_1} \cdots E_{\alpha_r}^{n_r} x : n_i \geq 0, x \in X \}$ spans $M$ as a vector space over $K$. If $x \in X$ has weight $\lambda_{x}$ then $E_{\alpha_1}^{n_1} \cdots E_{\alpha_r}^{n_r} x$ has weight $\lambda_{x} + \sum n_i \alpha_i$. For each weight $\mu \in \h^{*}$ and $x \in X$, $\mu - \lambda_{x}$ can only be written as a sum of positive roots in a finite number of ways, so the weight space $M_{\mu}$ is finite dimensional.
\end{proof}

Recall that $\n^k \subseteq \env(\n)$ denotes $\{ E_1 \cdots E_k : E_{i} \in \n$ for all $1 \leq i \leq k \}$ and $M^{\n^{\infty}}$ denotes all elements of $M$ which are annihilated by $\n^{k}$ for some $k$.

\begin{lemma}  \label{ninfty}
If $M \in \catobar$ then $M^{\vee} = (M^{*})^{\n^{\infty}}$.
\end{lemma}

\begin{proof}
For any $\env(\n)$-module $M$ let $\n^k M$ denote the smallest subspace of $M$ containing $E_1 \cdots E_k v$ for any $E_1, \cdots, E_k \in \n$ and $v \in M$. Let $\env(\n) \n^k$ denote the smallest subspace of $\env(\n)$ containing $u E_1 \cdots E_k$ for any $E_1, \cdots, E_k \in \n$ and $u \in \env(\n)$ (so $\n^k \env(\n) = \env(\n) \n^k$). If we pick an ordering $\{ \alpha_1, \cdots, \alpha_r \}$ of $\Phi^{+}$ then $\{ E_{\alpha_1}^{n_1} \cdots E_{\alpha_r}^{n_r} : n_j \geq 0 \}$ is a basis for $\env(\n)$. As $\env(\n) \n^k$ contains $\{ E_{\alpha_1}^{n_1} \cdots E_{\alpha_r}^{n_r} : \sum n_j \geq k \}$ it has finite codimension in $\env(\n)$.

Let $X$ be a finite set of weight vectors which generates $M$ as a $\env(\n)$-module, as constructed in the proof of Lemma~\ref{subcat}. If $x \in X$ has weight $\lambda_x$ then $E_{\alpha_{m_1}} \cdots E_{\alpha_{m_k}} x$ has weight $\lambda_x + \sum \alpha_{m_i}$ for any $\alpha_{m_1}, \cdots, \alpha_{m_k} \in \Phi$ (i.e. the order is not important), so for any weight $\lambda$ we can find $k \in \N$ such that $(\n^k M)_{\lambda} = 0$.

Let $\phi \in M^{\vee} = \bigoplus_{\lambda \in \h^{*}} (M^{*})_{\lambda}$. Say $\phi = \sum \phi_{\lambda}$ where $\phi_{\lambda} \in (M^{*})_{\lambda}$ for each $\lambda$ and the sum is over a finite set $\Lambda$. Then we can find $k \in \N$ such that $(\n^k M)_{- \lambda} = 0$ for all $\lambda \in \Lambda$, i.e. $\phi |_{\n^k M} = 0$. Then $\n^k \phi (M) = \phi (\n^k M) = 0$, so $\phi \in (M^{*})^{\n^{\infty}}$. Hence $M^{\vee} \subseteq (M^{*})^{\n^{\infty}}$.

Now suppose $\phi \in (M^{*})^{\n^{\infty}}$. Choose $k$ such that $\n^k \phi = 0$. Then $\env(\n) \n^k x \subseteq \ker \phi$ for all $x \in X$. The action of $\h$ preserves $\sum_{x \in X} \env(\n) \n^k x \subseteq M$, so it splits up into weight spaces. If we can show that the number of $\lambda$ such that $M_{\lambda} \nsubseteq \sum_{x \in X} \env(\n) \n^k x$ is finite then $\phi$ can only be non-zero on finitely many $M_{\lambda}$, and hence $\phi \in \bigoplus_{\lambda \in \h^{*}} (M^{*})_{\lambda}$. This would prove that $(M^{*})^{\n^{\infty}} \subseteq M^{\vee}$.

If we have $y \in M_{\lambda}$ then as $M = \sum_{x \in X} \env(\n) x$ we can write $y = \sum u_{x} x$ with each $u_{x} \in \env(\n)$. As each $x \in X$ is a weight vector, $\h$ acts diagonalisably on $\env(\n) x$, so $\env(\n) x = \bigoplus_{\mu \in \h^{*}} (\env(\n) x)_{\mu}$. We may thus replace each $u_{x} x$ with its component in $(\env(\n) x)_{\lambda}$ and still get $y = \sum u_{x} x$ with each $u_{x} x$ of weight $\lambda$. Thus we need to show that for each $x \in X$ there are only finitely many $\lambda$ such that $(\env(\n) x)_{\lambda} \nsubseteq \env(\n) \n^k x$. This follows from the fact that $\env(\n) \n^k$ has finite codimension in $\env(\n)$.
\end{proof}

\section{From maps between Verma modules to maps between the spaces $\mathcal{C}^{\la}_{\cs} (N, K_{\nu})$} \label{sdefmaps}

For $\nu \in \h^{*}$ let $A_{\nu}$ be the one-dimensional $K$-vector space with the action of $\env(\bbar)$ given by extending $\nu$ to $\bbar$ by letting $\nbar$ act trivially. For $\nu \in X(\mathbf{T})$ this is consistent with our earlier definition. We define
\[
\verma{\bbar}{\nu} = \env(\g) \tensor_{\env(\bbar)} A_{\nu}^{*}
\]
where $\env(\g)$ is a $\env(\bbar)$ module by multiplication on the right. This is a Verma module for the Borel subalgebra $\bbar$ (it is more standard to use $\b$). It is a lowest weight module, with lowest weight $- \nu$. It is in $\catobar$, and hence in $\mathcal{C}$ by Lemma~\ref{subcat}.
% check lemma ref?

\comment{
Let $\h^{*}$ denote the set of Lie algebra homomorphisms from $\h$ to $K$. For $\mu \in \h^{*}$ we write $\verma{\b}{\mu}$ to denote the Verma module over $K$ with highest weight~$\mu$, i.e. $\verma{\b}{\mu} = \env(\g) \tensor_{\env(\b)} K_{\nu}$, where $\env(\g)$ is a $\env(\b)$ module by multiplication on the right and we let $\n$ act trivially on $K_{\nu}$. This is a highest weight module, with highest weight $\mu$. It is in $\cato$. We also define
\[
\verma{\bbar}{\mu} = \env(\g) \tensor_{\env(\bbar)} (K_{\mu})^{*}
\]
where $K_{\mu}$ is the one-dimensional $K$-vector space with the action of $\env(\bbar)$ given by extending $\mu$ to $\bbar$ by letting $\mathfrak{\nbar}$ act trivially. This is a lowest weight module, with lowest weight $- \mu$. It is in $\catobar$. For reasons that will become clear in \S\ref{sBGG}, we are more interested in $\verma{\bbar}{\mu}$ than in $\verma{\b}{\mu}$.
}

Suppose $\lambda$ and $\mu$ are in $X(\mathbf{T})$ and fix a morphism $\psi: \verma{\bbar}{\mu} \to \verma{\bbar}{\lambda}$. Applying the functor
%% $F: \catobar \to \cato$,
$F: \mathcal{C} \to \mathcal{C}$, $M \to M^{\vee}$ from
%% CUT THE LEMMA REFERENCE
Lemma~\ref{exact} we get a map $\psi^{\vee}: \verma{\bbar}{\lambda}^{\vee} \to \verma{\bbar}{\mu}^{\vee}$ of $\env(\g)$-modules.

By Proposition~5.5.4 of \cite{dixmier}, since $\verma{\bbar}{\nu} = \env(\g) \tensor_{\env(\bbar)} A_{\nu}^{*}$ we have that $(\verma{\bbar}{\nu})^{*} \cong \Hom_{\env(\bbar)} (\env(\g), A_{\nu}^{**})$, where we make $\env(\g)$ a $\env(\bbar)$-module by the left regular representation and the action of $\env(\g)$ on $\Hom_{\env(\bbar)} (\env(\g), A_{\nu}^{**})$ is by multiplication on the right on the source, i.e. $u \phi (u') = \phi (u' u)$ for all $u, u' \in \env(\g)$ and $\phi \in \Hom_{\env(\bbar)} (\env(\g), A_{\nu}^{**})$. As $A_{\nu}$ is finite dimensional we have a natural isomorphism $A_{\nu}^{**} \cong A_{\nu}$, so
\[
\verma{\bbar}{\nu}^{*} \cong \Hom_{\env(\bbar)} (\env(\g), A_{\nu}).
\]
Let $\nu \in X(\mathbf{T})$. In \S\ref{ssubspaces} we defined an action of $\g$ on $\mathcal{C}^{\pol}(N, K_{\nu})$ by declaring the isomorphism of $\env(\n)$-modules $\mathcal{C}^{\pol}(N, K_{\nu}) \cong \Hom_{\env(\bbar)} (\env(\g), A_{\nu})^{\n^{\infty}}$ to be an isomorphism of $\env(\g)$-modules. Combining this with Lemma~\ref{ninfty} we get an isomorphism of $\env(\g)$-modules
\[
\zeta_{\nu}: \verma{\bbar}{\nu}^{\vee} \xrightarrow{\sim} \mathcal{C}^{\pol}(N, K_{\nu})
\]
by composing $\verma{\bbar}{\nu}^{\vee} \cong (\verma{\bbar}{\nu}^{*})^{\n^{\infty}} \cong \Hom_{\env(\bbar)} (\env(\g), A_{\nu})^{\n^{\infty}} \cong \mathcal{C}^{\pol}(N, K_{\nu})$.

We define $\psi^{\pol} : \mathcal{C}^{\pol} (N, K_{\lambda}) \to \mathcal{C}^{\pol} (N, K_{\mu})$, by $\psi^{\pol} = \zeta_{\mu}^{-1} \circ \psi^{\vee} \circ \zeta_{w \cdot \lambda}$. This is a morphism of $\env(\g)$-modules. Recall that $\mathcal{C}^{\lp}_{\cs} (N, K_{\nu}) = \mathcal{C}^{\pol} (N, K_{\nu}) \tensor_{K} \mathcal{C}^{\sm}_{\cs} (N, K)$. We define $\psi^{\lp}: \mathcal{C}^{\lp}_{\cs} (N, K_{\lambda}) \longrightarrow \mathcal{C}^{\lp}_{\cs} (N, K_{\mu})$ on simple tensors by
\[
\psi^{\lp} (f_1 \tensor f_2) = \psi^{\pol} (f_1) \tensor f_2
\]
and extend $K$-linearly. This is $\g$-equivariant, as $\g$ acts trivially on $\mathcal{C}^{\sm}_{\cs} (N, K)$. We spend the remainder of this section proving that $\psi^{\lp}$ extends to a unique $\g$-equivariant function $\psi^{\la}: \mathcal{C}^{\la}_{\cs} (N, K_{\lambda}) \to \mathcal{C}^{\la}_{\cs} (N, K_{\mu})$.

By the Poincar\'e-Birkhoff-Witt theorem, there is a unique $u_{\psi} \in \env(\n)$ such that $\psi(1 \tensor 1) = u_{\psi} \tensor 1$. Since $\verma{\bbar}{\mu} = \env(\g) \tensor_{\env(\bbar)} K_{\mu}$ is generated as a $\env(\g)$-module by the single element $1 \tensor 1$, $u_{\psi}$ determines $\psi$ by the formula $\psi (u \tensor 1) = u (\psi (1 \tensor 1)) = u u_{\psi} \tensor 1$.

Since $\zeta_{\nu}^{-1}: \mathcal{C}^{\pol} (N, K_{\nu}) \cong \verma{\bbar}{\nu}^{\vee}$ sends $f$ to the map $u \tensor 1 \mapsto S(u) f (e)$, for any $f \in \mathcal{C}^{\pol} (N, K_{\lambda})$ we have
\begin{equation} \label{thetaproperty}
S(u) \psi^{\pol} (f) (e) = S(u u_{\psi}) f (e)
\end{equation}
for all $u \in \env(\g)$. Moreover, the fact that $\zeta_{\mu}$ is an isomorphism implies that any $f' \in \mathcal{C}^{\pol} (N, K_{\mu})$ is determined by knowing $S(u) f'(e)$ for all $u \in \env(\g)$, so \eqref{thetaproperty} uniquely determines $\psi^{\pol} (f)$.

Let us examine the $\env(\g)$-action on $\mathcal{C}^{\lp}(N, K_{\lambda})$. By Lemma~2.5.24 of \cite{emertonjacquet2}, the natural map $\mathcal{C}^{\lp}_{\cs} (N, K_{\lambda}) \to \mathcal{C}^{\la}_{\cs} (N, K_{\lambda})$ given by multiplication of polynomial functions by smooth functions is a $\env(\g)$-equivariant injection. The $\env(\g)$-action on $\mathcal{C}^{\la}_{\cs} (N, K_{\lambda})$ is given by the isomorphism $\mathcal{C}^{\la}_{\cs} (N, K_{\lambda}) \to \Ind_{\Bbar}^{G} (\lambda)(N)$. The $\env(\g)$-action on $\Ind_{\Bbar}^{G}(\lambda)(N)$ comes from differentiating the right regular action of $G$ on $\mathcal{C}^{\la}(G, K)$.

But there is also the left regular action of $G$ on $\mathcal{C}^{\la}(G, K)$, which we denote with a subscript L. Recall that to make this a left action rather than a right action we define it as
\[
h_{L} f(g) = f(h^{-1} g).
\]
We can also differentiate the left regular action of $G$ to get an action of $\g$, which we call the L action to distinguish it from our original action of $\g$. Since the left and right regular actions of $G$ on $\mathcal{C}^{\la} (G, K)$ commute, the L action of $\g$ commutes with the right regular action of $G$.

For any $g \in G$ and $f \in \mathcal{C}^{\la} (G, K)$ we have
\[
(g_{L} f) (e) = f(g^{-1}) = (g^{-1} f) (e)
\]
so for any $X \in \g$ we have $X_{L} f (e) = (-X) f (e)$, and hence $u_{L} f (e) = S(u) f (e)$ for any $u \in \env(\g)$. It follows that for $u \in \env(\g)$ and $f \in \mathcal{C}^{\la} (G, K)$ we have
\begin{equation} \label{eight}
\begin{split}
S(u) ( (u_{\psi})_L f) (e) & = (u_{\psi})_L (S(u) f) (e)
\\ & = S(u_{\psi}) (S(u) f) (e)
\\ & = S(u u_{\psi}) f (e)
\end{split}
\end{equation}
which closely resembles \eqref{thetaproperty}.

We now establish some properties of the L action of $\env(\g)$. Let $\Omega \subseteq G$ be closed and open, $f \in \mathcal{C}^{\la} (G, K)$ and $u \in \env(\g)$. Recall that $f_{| \Omega} = f \mathbf{1}_{\Omega}$.

\begin{lemma} \label{uLrestrict}
We have $u_L (f_{| \Omega}) = (u_L f)_{| \Omega}$.
\end{lemma}

\begin{proof}
For any $X \in \g$, $X_{L} (f \mathbf{1}_{\Omega}) = (X_{L} f) \mathbf{1}_{\Omega} + f (X_{L} \mathbf{1}_{\Omega})$, by the Leibniz rule. But $X_{L} \mathbf{1}_{\Omega} = 0$ as $\mathbf{1}_{\Omega}$ is smooth, so $X_{L} (f \mathbf{1}_{\Omega}) = (X_{L} f) \mathbf{1}_{\Omega}$. It follows that $u_{L} (f \mathbf{1}_{\Omega}) = (u_{L} f) \mathbf{1}_{\Omega}$ for all $u \in \env(\g)$, and hence that $u_L (f_{| \Omega}) = (u_L f)_{| \Omega}$.
\end{proof}

\begin{corollary} \label{preservessupport}
If $g = u_L f$ then $\supp g \subseteq \supp f$, and we can find $f' \in \mathcal{C}^{\la} (G, K)$ such that $u_L f' = g$ and $\supp g = \supp f'$.
\end{corollary}

\begin{proof}
Since $f = f_{| \supp f}$ we have that $g = u_L f = u_L (f_{| \supp f}) = (u_L f)_{| \supp f} = g_{| \supp f}$, whence it follows that $\supp g \subseteq \supp f$.

Set $f' = f_{| \supp g}$, so by construction $\supp f' = \supp g$. Then $u_L f' = u_L (f_{| \supp g}) = (u_L f)_{| \supp g} = g_{| \supp g} = g$.
\end{proof}

Suppose further that $\Omega$ is locally analytically isomorphic to the $L$-points of a rigid analytic space in a way which is compatible with all charts of $G$.

\begin{lemma} \label{preservesanalytic}
If $f$ is analytic on $\Omega$ then $(u_L f)$ is also analytic on $\Omega$.
\end{lemma}

\begin{proof}
The L action of $\env(\g)$ on $\mathcal{C}^{\la} (G, K)$, and hence on $\mathcal{C}^{\an} (\Omega, K)$, is via differential operators, which preserve analytic functions.
\end{proof}

Now we are in a position to define $\psi^{\la}$.

\begin{theorem} \label{psiextends}
There is a unique continuous $\env(\g)$-module homomorphism $\psi^{\la}: \mathcal{C}^{\la}_{\cs}(N, K_{\lambda}) \to \mathcal{C}^{\la}_{\cs}(N, K_{\mu})$ extending $\psi^{\lp}$.
\end{theorem}

\begin{proof}
Uniqueness is immediate by the density of $\mathcal{C}^{\lp}_{\cs} (N, K_{\lambda})$ in $\mathcal{C}^{\la}_{\cs}(N, K_{\lambda})$.

Let $f_1 \in \mathcal{C}^{\pol} (N, K_{\lambda})$ and $f_2 \in \mathcal{C}^{\sm}_{\cs} (N, K)$. We identify $\mathcal{C}^{\lp}_{\cs} (N, K_{\nu})$ with its image in $\mathcal{C}^{\la}_{\cs} (N, K_{\nu})$ and we define
\[
\Phi_{\nu}: \mathcal{C}^{\la}_{\cs} (N, K_{\nu}) \longrightarrow \mathcal{C}^{\la} (G, K) (\Bbar N)
\]
to be the continuous, $\env(\g)$-equivariant inclusion obtained by composing the isomorphism $\mathcal{C}^{\la}_{\cs} (N, K_{\nu}) \cong \Ind_{B}^{G} (\nu) (N)$ with the inclusion $\Ind_{B}^{G} (\nu) (N) \subseteq \mathcal{C}^{\la} (G, K) (\Bbar N)$. This means that $\Phi_{\nu} (f_1 \tensor f_2)$ is defined on $\Bbar N$ by $\overline{b} n \mapsto \nu (\overline{b}) f_1(n) f_2(n)$ and is $0$ outside of $\Bbar N$. In particular, $\Phi_{\nu} (f_1 \tensor f_2)(e) = f_1(e) f_2(e)$.

Let us examine $F = (u_{\psi})_L \Phi_{\lambda}(f_1 \tensor f_2) - \Phi_{\mu}(\psi^{\lp}(f_1 \tensor f_2))$. By \eqref{thetaproperty}, \eqref{eight} and the $\env(\g)$-equivariance of $\Phi_{\mu}$, for all $u \in \env(\g)$ we have
\begin{align*}
S(u) F (e) & = S(u) \left((u_{\psi})_L (\Phi_{\lambda}(f_1 \tensor f_2)) \right) (e) - S(u) ( \Phi_{\mu}(\psi^{\pol}(f_1) \tensor f_2) ) (e) \\
& = S ( u u_{\psi}) (\Phi_{\lambda}(f_1 \tensor f_2)) (e) - \Phi_{\mu}(S(u) \psi^{\pol}(f_1) \tensor f_2) (e)  \\
& = \Phi_{\lambda}(S ( u u_{\psi}) f_1 \tensor f_2) (e) - S(u) \psi^{\pol}(f_1)(e) f_2(e)  \\
& = S ( u u_{\psi}) f_1(e) f_2(e) -  S(u u_{\psi}) f_1 (e) f_2(e) \\
& = 0.
\end{align*}
We've just shown that $u F (e) = 0$ for all $u \in \env(\g)$, and hence that the image of $F$ under all point distributions at $e$ is $0$. It follows that $F$ must be identically 0 in some neighbourhood of $e$.

Let $X$ be a chart of $N$ containing $e$ and set $f_2 = \mathbf{1}_{X}$. Then $F \in \mathcal{C}^{\la} (G, K)$ has $\supp F \subseteq \Bbar X$ by Lemma~\ref{preservessupport}, and it is analytic on $\Bbar X$ by Lemma~\ref{preservesanalytic}. Hence it is $0$ on $\Bbar X$, and we have shown that $F = 0$.

Let $Y \subseteq N$ be any compact, open subset, and choose a chart $X \subseteq N$ containing $e$ such that $Y \subseteq X$. By the above argument we know that $\Phi_{\mu}(\psi^{\lp}(f_1 \tensor \mathbf{1}_{X})) = (u_{\psi})_L \Phi_{\lambda}(f_1 \tensor \mathbf{1}_{X})$. Then, using Lemma~\ref{uLrestrict} and the fact that $\Phi_{\mu} (g_{|Y}) = \Phi_{\mu}(g)_{|\Bbar Y}$, we have that
\begin{align*}
\Phi_{\mu}(\psi^{\lp}(f_1 \tensor \mathbf{1}_{Y})) & = \Phi_{\mu}((\psi^{\pol}(f_1) \tensor \mathbf{1}_{X})_{|Y}) \\
& = (\Phi_{\mu}(\psi^{\pol}(f_1) \tensor \mathbf{1}_{X}))_{|\Bbar Y} \\
& = ((u_{\psi})_L \Phi_{\lambda}(f_1 \tensor \mathbf{1}_{X}))_{|\Bbar Y} \\
& = (u_{\psi})_L (\Phi_{\lambda}(f_1 \tensor \mathbf{1}_{X})_{|\Bbar Y}) \\
& = (u_{\psi})_L \Phi_{\lambda}((f_1 \tensor \mathbf{1}_{X})_{|Y}) \\
& = (u_{\psi})_L \Phi_{\lambda}(f_1 \tensor \mathbf{1}_{Y})
\end{align*}
By linearity it follows that $\Phi_{\mu}(\psi^{\lp}(g)) = (u_{\psi})_L \Phi_{\lambda}(g)$ for all $g \in \mathcal{C}^{\lp}_{\cs} (N, K_{\lambda})$. In particular this implies that
\begin{align*}
(u_{\psi})_L \Phi_{\lambda}(\mathcal{C}^{\lp}_{\cs} (N, K_{\lambda})) & = \Phi_{\mu}(\psi^{\lp}(\mathcal{C}^{\lp}_{\cs} (N, K_{\lambda}))) \\
& \subseteq \Phi_{\mu}(\mathcal{C}^{\lp}_{\cs} (N, K_{\mu})) \\
& \subseteq \Ind_{B}^{G} (\mu) (N)
\end{align*}
All the maps involved are continuous, $\mathcal{C}^{\lp}_{\cs} (N, K_{\lambda})$ is dense in $\mathcal{C}^{\la}_{\cs} (N, K_{\lambda})$ and $\Ind_{B}^{G} (\mu) (N)$ is a closed subspace of $\mathcal{C}^{\la} (G, K) (\Bbar N)$ so it follows that $(u_{\psi})_L \Phi_{\lambda}(\mathcal{C}^{\la}_{\cs} (N, K_{\lambda})) \subseteq \Ind_{B}^{G} (\mu) (N)$.

Since $\Phi_{\mu} : \mathcal{C}^{\la}_{\cs} (N, K_{\mu}) \to \Ind_{B}^{G} (\mu) (N)$ is an isomorphism we can define
\[
\psi^{\la} = \Phi_{\mu}^{-1} \circ (u_{\psi})_L \circ \Phi_{\lambda}: \mathcal{C}^{\la}_{\cs}(N, K_{\lambda}) \to \mathcal{C}^{\la}_{\cs}(N, K_{\mu}).
\]
This is continuous and $\g$-equivariant as all the maps involved in its definition are. As $\Phi_{\mu}(\psi^{\lp}(f)) = (u_{\psi})_L \Phi_{\lambda}(f)$ for all $f \in \mathcal{C}^{\lp}_{\cs} (N, K_{\lambda})$, this extends $\psi^{\lp}$.
\end{proof}

\begin{lemma} \label{psiproperties}
Let $f \in \mathcal{C}^{\la}_{\cs}(N, K_{\lambda})$.
\begin{enumerate}
\item \label{partone} If $\Omega \subseteq N$ is open and closed then $\psi^{\la} (f_{| \Omega}) = \psi^{\la} (f)_{| \Omega}$.
\item \label{parttwo} $\supp \psi^{\la} (f) \subseteq \supp f$.
\item \label{partthree}$f' = f_{| \supp \psi^{\la} (f)}$ satisfies $\psi^{\la} (f') = \psi^{\la} (f)$ and $\supp f' = \supp \psi^{\la} (f)$.
\item \label{partfour} If $f$ is analytic on $X \subseteq N$ then so is $\psi^{\la} (f)$.
\item \label{partfive} For all $n \in N$, $\psi^{\la} (n f) = n \psi^{\la} (f)$.
\end{enumerate}
\end{lemma}

\begin{proof}
Parts \ref{partone}-\ref{partfour} follow immediately from Lemmas~\ref{uLrestrict}, \ref{preservessupport} and \ref{preservesanalytic}. Since the L action of $\g$ and the right regular action of $G$ on $\mathcal{C}^{\la} (G, K)$ commute, part \ref{partfive} follows from the fact that $\Phi_{\nu} (n f) = n \Phi_{\nu} (f)$.
\end{proof}

\begin{remark}
In fact all $\env(\g)$-equivariant maps $\mathcal{C}^{\la}_{\cs}(N, K_{\lambda}) \to \mathcal{C}^{\la}_{\cs}(N, K_{\mu})$ are of the form $\psi^{\la}$ for some $\psi: \verma{\bbar}{\mu} \to \verma{\bbar}{\lambda}$, although we will not use this fact.
\end{remark}

\comment{
\begin{theorem}
Any map $\phi: \mathcal{C}^{\la}_{\cs}(N, K_{\lambda}) \to \mathcal{C}^{\la}_{\cs}(N, K_{\mu})$ is of the form $\psi^{\la}$ for some $\psi: \verma{\bbar}{\mu} \to \verma{\bbar}{\lambda}$.
\end{theorem}

\begin{proof}
First
\end{proof}
}

\section{A BGG-type resolution} \label{sBGG}

In this section we assume that $\mathbf{G}$ is semisimple. When using results from \S\ref{sliealg} we work with the semisimple Lie algebra $\g_{K} = \Lie(\mathbf{G}(K))$, Cartan subalgebra $\h_{K} = \Lie(\mathbf{T}(K))$ (which is split because $\mathbf{T}$ is maximal and split over $K$) and Borel subalgebra $\b_{K} = \Lie(\mathbf{B}(K))$. Since representations of $\g$ over $K$ are exactly the same thing as representations of $\g \tensor_{L} K = \g_{K}$ over $K$ we will implicitly equate the two.

As $V^{*}$ is a finite dimensional irreducible algebraic representation of $G$, and hence of $\env(\g)$, over $K$, with lowest weight $-\lambda$, the \textbf{Bernstein-Gelfand-Gelfand resolution} of $V^{*}$ with respect to $\bbar$ is the exact sequence of $\env(\g)$-modules
\begin{multline} \label{bgg}
0 \to \verma{\bbar}{w_0 \cdot \lambda} \to \cdots \to \bigoplus_{w \in W^{(i)}} \verma{\bbar}{w \cdot \lambda} \to \\
\cdots \to \bigoplus_{w \in W^{(1)}} \verma{\bbar}{w \cdot \lambda} \to \verma{\bbar}{\lambda} \to V^{*} \to 0.
\end{multline}

\comment{
Recall from \S\ref{ssnotation} that $V$ is a finite dimensional irreducible algebraic representation of $G$, and hence of $\env(\g)$, over $K$, with highest weight $\lambda$. The \textbf{Bern{\v{s}}te{\u\i}n-Gel'fand-Gel'fand resolution} of $V$ is the exact sequence of $\env(\g)$-modules
\begin{multline*}
0 \to \verma{\b}{w_0 \cdot \lambda} \to \cdots \to \bigoplus_{w \in W^{(i)}} \verma{\b}{w \cdot \lambda} \to \\
\cdots \to \bigoplus_{w \in W^{(1)}} \verma{\b}{w \cdot \lambda} \to \verma{\b}{\lambda} \to V \to 0.
\end{multline*}
}

The BGG resolution was first constructed in \cite{bgg75} for a semisimple Lie algebra $\g$ over $\C$. A more recent treatment is given in \cite{humphreyscato}. As indicated at the beginning of \S0.1 of \cite{humphreyscato},  $\C$ is normally taken to be the field for convenience, but all that is required is that the field $K$ has characteristic 0 and $\h$ is a split Cartan subalgebra over $K$. It can be checked that the proof of the BGG resolution given in \cite{humphreyscato} holds in this case.

\comment{
We could equally well make this construction with a different choice of Borel subalgebra, which gives us a different resolution. Let us use $\bbar$ instead of $\b$. Consider the BGG resolution of $V^{*}$, this time in terms of the $\verma{\bbar}{\mu}$ instead of the $\verma{\b}{\mu}$. The lowest weight of $V^{*}$ is $-\lambda$, so we get:

\begin{multline} \label{bgg}
0 \to \verma{\bbar}{w_0 \cdot \lambda} \to \cdots \to \bigoplus_{w \in W^{(i)}} \verma{\bbar}{w \cdot \lambda} \to \\
\cdots \to \bigoplus_{w \in W^{(1)}} \verma{\bbar}{w \cdot \lambda} \to \verma{\bbar}{\lambda} \to V^{*} \to 0.
\end{multline}
}
With the exception of $\verma{\bbar}{\lambda} \to V^{*}$, the maps in \eqref{bgg} are of the form
\begin{align*}
\bigoplus_{w' \in W^{(i)}} \verma{\bbar}{w' \cdot \lambda} & \longrightarrow \bigoplus_{w \in W^{(i - 1)}} \verma{\bbar}{w \cdot \lambda} \\
(f_{w'})_{w' \in W^{(i)}} & \longmapsto \left( \sum_{\substack{\alpha \in \Phi^{+} \\ l(s_{\alpha} w) = i}} \theta_{\alpha, w \cdot \lambda} (f_{s_{\alpha} w}) \right)_{w \in W^{(i - 1)}}
\end{align*}
where $\theta_{\alpha, w \cdot \lambda}$ denotes a non-zero map $\verma{\bbar}{s_{\alpha} w \cdot \lambda} \to \verma{\bbar}{w \cdot \lambda}$.

\comment{
This is an exact sequence in $\catobar$, so we can apply the contravariant exact functor $F: \catobar \to \cato, M \mapsto M^{\vee}$ to get the following exact sequence in $\cato$:
}
This is an exact sequence in $\mathcal{C}$, so we can apply the contravariant exact functor $F: \mathcal{C} \to \mathcal{C}, M \mapsto M^{\vee}$ from Lemma~\ref{exact} to get the following exact sequence in $\mathcal{C}$:
\begin{multline} \label{dualbgg}
0 \to V \to \verma{\bbar}{\lambda}^{\vee} \to \bigoplus_{w \in W^{(1)}} \verma{\bbar}{w \cdot \lambda}^{\vee} \to \\
\cdots \to \bigoplus_{w \in W^{(i)}} \verma{\bbar}{w \cdot \lambda}^{\vee} \to \cdots \to \verma{\bbar}{w_0 \cdot \lambda}^{\vee} \to 0
\end{multline}
Using the isomorphisms $\zeta_{w \cdot \lambda}: \verma{\bbar}{w \cdot \lambda}^{\vee} \xrightarrow{\sim} \mathcal{C}^{\pol}(N, K_{w \cdot \lambda})$ we rewrite \eqref{dualbgg} as
\begin{multline} \label{polES}
0 \to V \to \mathcal{C}^{\pol}(N, K_{\lambda}) \to \bigoplus_{w \in W^{(1)}} \mathcal{C}^{\pol}(N, K_{w \cdot \lambda}) \to \\
\cdots \to \bigoplus_{w \in W^{(i)}} \mathcal{C}^{\pol}(N, K_{w \cdot \lambda}) \to \cdots \to \mathcal{C}^{\pol}(N, K_{w_0 \cdot \lambda}) \to 0
\end{multline}

This is an exact sequence in $\env(\g)$-$\Mod$. Recall that $\mathcal{C}^{\lp}_{\cs}(N, K_{\nu})$ is defined to be $\mathcal{C}^{\pol}(N, K_{\nu}) \tensor_{K} \mathcal{C}^{\sm}_{\cs}(N, K)$. We tensor \eqref{polES} over $K$ with $\mathcal{C}^{\sm}_{\cs}(N, K)$, which is considered as an object in $\env(\g)$-$\Mod$ by letting $\g$ act trivially on it. This preserves exactness, as any module over a field is flat and exactness is a property only of the underlying sequence of vector spaces. Thus we get the exact sequence of $\env(\g)$-modules:
\begin{multline} \label{lpES}
0 \to V \tensor_{K} \mathcal{C}^{\sm}_{\cs}(N, K) \to \mathcal{C}^{\lp}_{\cs}(N, K_{\lambda}) \to \bigoplus_{w \in W^{(1)}} \mathcal{C}^{\lp}_{\cs}(N, K_{w \cdot \lambda}) \to \\
\cdots \to \bigoplus_{w \in W^{(i)}} \mathcal{C}^{\lp}_{\cs}(N, K_{w \cdot \lambda}) \to \cdots \to \mathcal{C}^{\lp}_{\cs}(N, K_{w_0 \cdot \lambda}) \to 0
\end{multline}
With the exception of $V \tensor_{K} \mathcal{C}^{\sm}_{\cs}(N, K) \to \mathcal{C}^{\lp}_{\cs}(N, K_{\lambda})$, the maps in \eqref{lpES} are of the form
\begin{align*}
\bigoplus_{w \in W^{(i - 1)}} \mathcal{C}^{\lp}_{\cs}(N, K_{w \cdot \lambda}) & \longrightarrow \bigoplus_{w' \in W^{(i)}} \mathcal{C}^{\lp}_{\cs}(N, K_{w' \cdot \lambda}) \\
(f_{w})_{w \in W^{(i - 1)}} & \longmapsto \left( \sum_{\substack{\alpha \in \Phi^{+} \\ l(s_{\alpha} w') = i - 1}} \theta_{\alpha, s_{\alpha} w' \cdot \lambda}^{\lp} (f_{s_{\alpha} w'}) \right)_{w' \in W^{(i)}}
\end{align*}

Using the same formulae with $\theta_{\alpha, s_{\alpha} w' \cdot \lambda}^{\lp}$ replaced with $\theta_{\alpha, s_{\alpha} w' \cdot \lambda}^{\la}$ we get the following sequence of $\env(\g)$-modules:
\begin{multline} \label{laES}
0 \to V \tensor_{K} \mathcal{C}^{\sm}_{\cs} (N, K) \xrightarrow{d_{-1}} \mathcal{C}^{\la}_{\cs} (N, K_{\lambda}) \xrightarrow{d_{0}} \bigoplus_{w \in W^{(1)}} \mathcal{C}^{\la}_{\cs} (N, K_{w \cdot \lambda})  \xrightarrow{d_{1}} \\
\cdots \xrightarrow{d_{i - 1}} \bigoplus_{w \in W^{(i)}} \mathcal{C}^{\la}_{\cs} (N, K_{w \cdot \lambda}) \xrightarrow{d_{i}} \cdots \xrightarrow{d_{r - 1}} \mathcal{C}^{\la}_{\cs} (N, K_{w_0 \cdot \lambda}) \to 0
\end{multline}
We will prove that this sequence is exact in Corollary~\ref{laEScor}.

\section{Reductive groups} \label{sredgps}

Now let $\mathbf{G}$ be reductive and let $\mathbf{G}'$ denote the derived subgroup of $\mathbf{G}$, which is defined over $L$ by the first Corollary in \S2.3 of \cite{BorelLAG}. Note that $\mathbf{G}'$ is semisimple and $\mathbf{T}' = \mathbf{T} \cap \mathbf{G}'$ is a maximal torus in $\mathbf{G}'$ and split over $K$. We therefore have the sequence \eqref{laES} for $\mathbf{G}'$, and since $\mathbf{N} = \mathbf{N} \cap \mathbf{G}'$ this almost gives us sequence \eqref{laES} for $\mathbf{G}$. The problem is that we only know that the maps are equivariant for the action $\g' = \Lie(G')$.

\begin{theorem} \label{laESred}
The maps in \eqref{laES} are $\g$-equivariant.
\end{theorem}

\begin{proof}
Let $\mathbf{Z}$ denote the center of $\mathbf{G}$. It is defined over $L$ by 12.1.7(b) of \cite{SpringerLAG}. Write $Z$ for $\mathbf{Z}(L)$ and $\mathfrak{z}$ for the corresponding Lie subalgebra of $\g$. Since $\mathbf{G} = \mathbf{G}' \mathbf{Z}$ we have $G = G' Z$, whence it follows that $\g = \g' + \mathfrak{z}$. Since $\mathbf{Z}$ centralises $\mathbf{T}$ the Weyl groups for $\mathbf{G}'$ and $\mathbf{G}$ are canonically isomorphic. Applying the results of the last section to $\mathbf{G}'$ we get the required sequence of maps of $\env(\g')$-modules. It only remains to show that the maps are morphisms of $\env(\g)$-modules, and to do this we need only check that the maps are equivariant for the action of $\mathfrak{z}$.

The action of $T$ on a highest weight vector $v$ of $V$ is via $\lambda$. Since $V$ is an irreducible representation of $G$, the set $\{ g v: g \in G \}$ spans $V$ (over $K$). Since $Z$ is contained in the centre of $G$ the action of $Z$ on $g v$ is via $\lambda$, as $z g v = g z v = \lambda(z) g v$. Hence $Z$ acts on all of $V$ via $\lambda$. Since $\mathcal{C}^{\sm}_{\cs}(N, K)$ has the trivial action of $\g$, the action of $\mathfrak{z}$ on $V \tensor_{K} \mathcal{C}^{\sm}_{\cs}(N, K)$ is via $\lambda$.

Let us now consider the action of $Z$ on $\mathcal{C}^{\la}_{\cs}(N, K_{w \cdot \lambda})$. For $z \in Z$, $x \in N$ and $f \in \mathcal{C}^{\la}_{\cs}(N, K_{w \cdot \lambda})$ we have
\[
(z f) (g) = (w \cdot \lambda)(z) f (z^{-1} g z) = (w \cdot \lambda)(z) f(g)
\]
So $Z$ acts through $w \cdot \lambda : T \to K^{\times}$, and hence $\mathfrak{z}$ acts through $w \cdot \lambda : \h \to K$.

The action of $W$ on $X(\mathbf{T})$ comes from the conjugation action of $\mathbf{N_{G} (T)}$ on $\mathbf{T}$, which is trivial on $\mathbf{Z} \subseteq \mathbf{T}$. Hence $\lambda|_{Z} = (w \cdot \lambda)|_{Z}$ for all $w \in W$, and thus all the maps in the sequence are $\mathfrak{z}$-equivariant.
\end{proof}

\section{Exactness of the First Three Terms} \label{sfirstthreeterms}

Fix $\mu \in X(\mathbf{T})$. We will now construct a basis for $\mathcal{C}^{\pol} (N, K_{\mu})$ which diagonalises the action of $\h$.

\begin{theorem} \label{polbasis}
We can find $T_1, \cdots, T_r \in \mathcal{C}^{\pol} (N, K_{\mu})$ such that $\mathcal{C}^{\pol} (N, K_{\mu}) \cong K[T_1, \cdots, T_r]$ and $T_1^{m_1} \cdots T_r^{m_r}$ is a weight vector of weight $\mu - \sum m_i \alpha_i$.
\end{theorem}

\begin{proof}
Recall we have an isomorphism of $\env(\g)$-modules $\zeta_{\mu}: \verma{\bbar}{\mu}^{\vee} \to \mathcal{C}^{\pol}(N, K_{\mu})$. Fix an ordering $\alpha_1, \cdots, \alpha_r$ of $\Phi^{+}$ and write $E_i$ for $E_{\alpha_i}$. By the Poincar\'{e}-Birkhoff-Witt theorem $\{ E_1^{n_1} \cdots E_r^{n_r} \tensor 1 : n_{i} \geq 0$ for all $i \}$ is a basis for $\verma{\bbar}{\mu}$. Let $\{ \eps_{m_1, \cdots, m_r} : m_{i} \geq 0$ for all $i \}$ be the dual basis for $\verma{\bbar}{\mu}^{\vee}$, defined by
\[
\eps_{m_1, \cdots, m_r} (E_1^{n_1} \cdots E_r^{n_r} \tensor 1) =
\begin{cases}
1 & \text{if $m_i = n_i$ for all $i$} \\
0 & \text{else.}
\end{cases}
\]
Since $E_1^{n_1} \cdots E_r^{n_r} \tensor 1$ has weight $-\mu + \sum n_i \alpha_i$, it follows that $\eps_{m_1, \cdots, m_r}$ has weight $\mu - \sum m_i \alpha_i$. (Recall that $X \in \g$ acts on $\phi \in \verma{\bbar}{\mu}^{\vee}$ via $X \phi (u \tensor 1) = \phi (-X u \tensor 1)$ for all $u \tensor 1 \in \verma{\bbar}{\mu}$.)

We define $T_i \in \mathcal{C}^{\pol} (N, K_{\mu})$ by $T_i = \zeta_{\mu} (\eps_{0, \cdots, 0, 1, 0, \cdots, 0})$, where the 1 is in the $i$th place. Using Lemmas~\ref{menial1} and \ref{menial2} which follow this proof we see that $\zeta_{\mu} (\eps_{m_1, \cdots, m_r}) = T_1^{m_1} \cdots T_r^{m_r} / m_1 ! \cdots m_r !$, so $\mathcal{C}^{\pol} (N, K_{\mu}) = K[T_1, \cdots, T_r]$, and $T_1^{m_1} \cdots T_r^{m_r}$ has weight $\mu - \sum m_i \alpha_i$ since $\zeta_{\mu}$ is $\env(\g)$-equivariant.
\end{proof}

\begin{remark}
If $\mu = w \cdot \lambda$ and $\alpha_{r} \in \Delta$ such that $l(s_{\alpha_{r}} w) = l(w) + 1$ then $\theta_{\alpha_{r}, w \cdot \lambda}^{\la}$ is a non-zero scalar multiple of $(\frac{\partial}{\partial T_r} )^{w \cdot \lambda(H_{\alpha}) + 1}$.
\end{remark}

Here are the two lemmas about $\zeta_{\mu}$ which were used in the proof.

\begin{lemma} \label{menial1}
$\zeta_{\mu} (\eps_{m_1, \cdots, m_r}) = \zeta_{\mu} (\eps_{m_1, 0, \cdots, 0}) \zeta_{\mu} (\eps_{0, m_2, 0, \cdots, 0}) \cdots \zeta_{\mu} (\eps_{0, \cdots, 0, m_r})$
\end{lemma}

\begin{proof}
The Leibniz rule says that for any $X, Y \in \env(\n)$ and $f, g \in \mathcal{C}^{\pol} (N, K_{\mu})$ we have $(XY fg)(e) = (XY f)(e) g(e) + (X f)(e) (Y g)(e) + (Y f)(e) (X g)(e) + f(e) (XY g)(e)$.
By repeated applications of this rule we may conclude that $S(E_r^{n_r} \cdots E_1^{n_1}) ( \zeta_{\mu} (\eps_{m_1, 0, \cdots, 0}) \zeta_{\mu} (\eps_{0, m_2, 0, \cdots, 0}) \cdots \zeta_{\mu} (\eps_{0, \cdots, 0, m_r}) )(e) = 0$ unless $m_i = n_i$ for all $i$, in which case it equals $1$. This is the defining characteristic of $\zeta_{\mu} (\eps_{m_1, \cdots, m_r})$.
\end{proof}

\begin{lemma} \label{menial2}
For all $m \geq 1$ we have $\zeta_{\mu} (\eps_{0, \cdots, m, \cdots, 0}) = \frac{1}{m!} \zeta_{\mu} (\eps_{0, \cdots, 1, \cdots, 0})^m$, where all the indices are $0$ except the $i$th.
\end{lemma}

\begin{proof}
We proceed by induction on $m$. For $m = 1$ the result is trivial. Suppose it is true for $m - 1$. By the same argument as in Lemma~\ref{menial1} we have that $S(E_r^{n_r} \cdots E_1^{n_1}) ( \zeta_{\mu} (\eps_{0, \cdots, 1, \cdots, 0}) \zeta_{\mu} (\eps_{0, \cdots, m - 1, \cdots, 0}) ) (e) = 0$ unless $n_i = m$ and $n_j = 0$ for all $j \neq i$, in which case it is
\[
\binom{m}{1} S(E_i) ( \zeta_{\mu} (\eps_{0, \cdots, 1, \cdots, 0} ) ) (e) \ S(E_i^{m - 1}) ( \zeta_{\mu} (\eps_{0, \cdots, m - 1, \cdots, 0} ) ) (e)
\]
which is $m$. Hence
\begin{align*}
\zeta_{\mu} (\eps_{0, \cdots, m, \cdots, 0}) & = \frac{1}{m} \zeta_{\mu} (\eps_{0, \cdots, 1, \cdots, 0} ) \ \zeta_{\mu} ( \eps_{0, \cdots, m - 1, \cdots, 0}) \\
& = \frac{1}{m} \zeta_{\mu} (\eps_{0, \cdots, 1, \cdots, 0} ) \ \frac{1}{(m - 1)!} \zeta_{\mu} (\eps_{0, \cdots, 1, \cdots, 0})^{m-1} \\
& = \frac{1}{m!} \zeta_{\mu} (\eps_{0, \cdots, 1, \cdots, 0})^m. \qedhere
\end{align*}
\end{proof}

Let $\mathbb{B}$ denote the rigid analytic closed unit ball defined over $L$. Let $X$ be a compact, open subset of $N$ such that there is a locally analytic isomorphism $X \cong \mathbb{B}(L)$ which is compatible with all charts of $N$. We write $\mathcal{C}^{\pol} (X, K_{\mu})$ for the subspace of $\mathcal{C}^{\la} (X, K_{\mu})$ given by restricting functions in $\mathcal{C}^{\pol} (N, K_{\mu})$ to $X$. Then $\mathcal{C}^{\pol} (X, K_{\mu}) \subset \mathcal{C}^{\an} (X, K_{\mu})$ and it inherits the subspace norm coming from the Gauss norm.

Since $\g$ acts on $\mathcal{C}^{\la}_{\cs} (N, K_{\mu})$ by differential operators, $\mathcal{C}^{\la} (N, K_{\mu})(X)$ is an $\env(\g)$-invariant subspace. Using the natural isomorphism we transfer this action of $\env(\g)$ to $\mathcal{C}^{\la} (X, K_{\mu})$, and to its $\env(\g)$-invariant subspaces $\mathcal{C}^{\an} (X, K_{\mu})$ and $\mathcal{C}^{\pol} (X, K_{\mu})$. This makes the map
\begin{align*}
\mathcal{C}^{\pol} (N, K_{\mu}) \longrightarrow \mathcal{C}^{\pol} (X, K_{\mu}) && f \longmapsto f|_{X}
\end{align*}
an isomorphism of $\env(\g)$-modules.

We now use the basis we have just constructed for $\mathcal{C}^{\pol} (N, K_{\mu})$ to study $\mathcal{C}^{\pol} (X, K_{\mu})$.

\begin{lemma} \label{sum}
Any $f \in \mathcal{C}^{\an} (X, K_{\mu})$ can be written uniquely as $\sum_{\nu \in Z_{\mu}} f_{\nu}$ where $Z_{\mu} \subseteq \h^{*}$ is the (countable) set of weights of $\mathcal{C}^{\pol} (X, K_{\mu})$ and $f_{\nu} \in \mathcal{C}^{\pol} (X, K_{\mu})_{\nu}$ for each $\nu \in Z_{\mu}$.
\end{lemma}

\begin{proof}
Choose $T_1, \cdots, T_r \in \mathcal{C}^{\pol} (N, K_{\mu})$ as in Theorem~\ref{polbasis}, but rescale them so that $\sup \{T_i(x) : x \in X \} = 1$. Replacing each $T_i$ with its restriction to $X$ we get $\mathcal{C}^{\pol} (X, K_{\mu}) = K[T_1, \cdots, T_r]$, so $Z_{\mu} = \{ \mu - \sum n_i \alpha_i : n_i \geq 0$ for all $i \}$. Then $\mathcal{C}^{\an} (X, K_{\mu})$ is the affinoid algebra
\[
K \langle T_1, \cdots, T_r \rangle = \left\{ \sum a_n T_1^{n_1} \cdots T_r^{n_r}: |a_n| \to 0 \text{ as } n_1 + \cdots + n_r \to \infty \right\}
\]
with norm $\| \sum a_n T_1^{n_1} \cdots T_r^{n_r} \| = \sup |a_n|$. Given $f = \sum_{n} a_n T_1^{n_1} \cdots T_r^{n_r} \in K \langle T_1, \cdots, T_r \rangle$, for any $\nu \in Z_{\mu}$ we define $f_{\nu} = \displaystyle \sum_{n: \sum n_i = \mu - \nu} a_n T_1^{n_1} \cdots T_r^{n_r}$. This is in $\mathcal{C}^{\pol} (X, K_{\mu})_{\nu}$, and we clearly have $f = \sum_{\nu \in Z_{\mu}} f_{\nu}$.

Suppose $f = \sum_{\nu \in Z_{\mu}} f_{\nu} = \sum_{\nu \in Z_{\mu}} f'_{\nu}$, with $f_{\nu}$ and $f'_{\nu} \in \mathcal{C}^{\pol} (X, K_{\mu})_{\nu}$. Then $0 = \sum_{\nu \in Z_{\mu}} (f_{\nu} - f'_{\nu})$ and by considering coefficients of the $T_1^{n_1} \cdots T_r^{n_r}$ we see that $f_{\nu} = f'_{\nu}$ for all $\nu \in Z_{\mu}$. Thus the expression is unique.
\end{proof}

\begin{corollary} \label{cortosum}
For all $\mu, \eta \in X(\mathbf{T})$ we have $\mathcal{C}^{\an} (X, K_{\mu})_{\eta} = \mathcal{C}^{\pol} (X, K_{\mu})_{\eta}$
\end{corollary}

\begin{proof}
Clearly $\mathcal{C}^{\pol} (X, K_{\mu})_{\eta} \subseteq \mathcal{C}^{\an} (X, K_{\mu})_{\eta}$. Let $f$ be a non-zero element in $\mathcal{C}^{\an} (X, K_{\mu})_{\eta}$. By Lemma~\ref{sum} $f = \sum_{\nu \in Z_{\mu}} f_{\nu}$, so for all $Y \in \h$ we have $0 = \eta(Y) f - Y f = \sum_{\nu \in Z_{\mu}} (\eta(Y) - \nu(Y)) f_{\nu}$. Since the unique expression for $0$ is $\sum 0$ we must have $(\eta(Y) - \nu(Y)) f_{\nu} = 0$ for all $\nu \in Z_{\mu}$. Hence for all $\nu \neq \eta$ we have $f_{\nu} = 0$, and thus we must have $\eta \in Z_{\mu}$ and $f = f_{\eta} \in \mathcal{C}^{\pol} (X, K_{\mu})_{\eta}$.
\end{proof}

\begin{lemma} \label{anESlemma}
We get an exact sequence
\[
0 \to V \xrightarrow{\delta_{-1}} \mathcal{C}^{\an} (X, K_{\lambda}) \xrightarrow{\delta_{0}} \bigoplus_{w \in W^{(1)}} \mathcal{C}^{\an} (X, K_{w \cdot \lambda})
\]
by restricting the first three terms of \eqref{laES}.
\end{lemma}

\begin{proof}
We define $\delta_{-1}$ by $v \mapsto d_{-1} (v \tensor \mathbf{1}_{X})|_{X}$. Let $\phi$ denote the map $V \to \mathcal{C}^{\pol}(N, K_{\lambda})$ from \eqref{polES}, so $d_{-1} (\sum v_i \tensor \mathbf{1}_{X_i}) = \sum \phi(v_i)_{| X_i}$. It follows that $\delta_{-1}(v) = \phi(v)|_{X}$, from which it is easily seen that $\delta_{-1}$ is well-defined, injective and has $\im \delta_{-1} \subseteq \mathcal{C}^{\pol} (X, K_{\lambda})$.

Given $f \in \mathcal{C}^{\an} (X, K_{\mu})$ we can extend it by $0$ to get $\overline{f} \in \mathcal{C}^{\la}_{\cs} (N, K_{\mu})$. We define $\delta_{0}$ by sending $f \mapsto d_{0}(\overline{f})|_{X}$, where each component is restricted to $X$. As $d_0 = \bigoplus_{\alpha \in \Delta} \theta_{\alpha, \lambda}^{\la}$, this is well-defined by Lemma~\ref{psiproperties}.\ref{partfour}. We know that $d_0 \circ d_{-1} = 0$ from the exactness of \eqref{lpES}, so for any $v \in V$
\[
\delta_{0} (\delta_{-1}(v)) = d_{0}( d_{-1}(v \tensor \mathbf{1}_{X})_{| X}) = d_{0}( d_{-1}(v \tensor \mathbf{1}_{X}))_{| X} = 0_{| X} = 0
\]
using Lemma~\ref{psiproperties}.\ref{partone}. Thus $\delta_{0} \circ \, \delta_{-1} = 0$, and it only remains to show that $\ker \delta_{0} \subseteq \im \delta_{-1}$.

Let $f \in \ker \delta_{0}$ and suppose we can show that $\ker \delta_{0} \subseteq \mathcal{C}^{\pol} (X, K_{\lambda})$. Then $\overline{f} \in \ker d_{0}$ is in $\mathcal{C}^{\lp}_{\cs} (N, K_{\lambda})$, so by exactness of \eqref{lpES} we can find $\sum v_i \tensor \mathbf{1}_{X_i} \in V \tensor \mathcal{C}^{\sm}_{\cs} (N, K)$ such that $d_{-1} (\sum v_i \tensor \mathbf{1}_{X_i}) = \overline{f}$. We may assume that the $X_i$ are disjoint charts of $N$ and that $X_i \subseteq X$ for all $i$. Let us compare $d_{-1} (\sum v_i \tensor \mathbf{1}_{X_i}) = \overline{f}$ with $d_{-1} (v_1 \tensor \mathbf{1}_{X}) = \phi(v_1)_{| X}$. They are both analytic on $X$ and they agree on the open subset $X_1$, so they must agree on all of $X$. Hence $f = \phi(v_1)|_{X} = \delta_{-1}(v_1)$ and we've shown that $f \in \im \delta_{-1}$.

To complete the proof it suffices to show that $\ker \delta_{0} \subseteq \mathcal{C}^{\pol} (X, K_{\lambda})$. Fix $f \in \ker \delta_{0}$. As $\delta_{0} = \bigoplus \theta_{\alpha, \lambda}$ where the sum is over all simple roots $\alpha$, this means $\theta_{\alpha, \lambda} (f) = 0$ for all $\alpha \in \Delta$. By Lemma~\ref{sum} for $\mu = \lambda$ we can write $f$ as $\sum_{\nu \in Z_{\lambda}} f_{\nu}$ with $f_{\nu} \in \mathcal{C}^{\pol}(X, K_{\lambda})_{\nu}$. Now for any $\alpha \in \Delta$, $\theta_{\alpha, \lambda}$ preserves weights and $\theta_{\alpha, \lambda} (\sum_{\nu \in Z_{\lambda}} f_{\nu}) = \sum_{\nu \in Z_{\lambda}} \theta_{\alpha, \lambda}(f_{\nu})$, so by the uniqueness of Lemma~\ref{sum} for $\mu = s_{\alpha} \cdot \lambda$ we must have $\theta_{\alpha, \lambda} (f_{\nu}) = 0$ for each $\nu \in Z_{\lambda}$. This is true for all simple roots, so $\delta_{0} (f_{\nu}) = 0$. By the exactness of \eqref{lpES} we can therefore find $v_{\nu} \in V$ such that $\delta_{-1} (v_{\nu}) = f_{\nu}$. In fact, as $\delta_{-1}$ is $\env(\g)$-equivariant we must have $v_{\nu} \in V_{\nu}$. But since $V$ is finite dimensional $V_{\nu} = 0$ for all but finitely many weights. Hence $f_{\nu} = 0$ for all but finitely many weights, and so $f \in \bigoplus_{\nu \in Z_{\lambda}} \mathcal{C}^{\pol} (X, K_{\lambda})_{\nu} = \mathcal{C}^{\pol} (X, K_{\lambda})$.
\end{proof}

\begin{theorem} \label{laESthreetermthm}
The first three terms of \eqref{laES}
\[
0 \to V \tensor_{K} \mathcal{C}^{\sm}_{\cs} (N, K) \xrightarrow{d_{-1}} \mathcal{C}^{\la}_{\cs} (N, K_{\lambda}) \xrightarrow{d_{0}} \bigoplus_{w \in W^{(1)}} \mathcal{C}^{\la}_{\cs} (N, K_{w \cdot \lambda})
\]
form an exact sequence.
\end{theorem}

\begin{proof}
It follows from the exactness of \eqref{lpES} that $d_{-1}$ is injective and $d_{0} \circ d_{-1} = 0$. It only remains to prove that $\ker d_{0} \subseteq \im d_{-1}$.

Let us fix $f \in \ker d_{0}$. By the definition of $f$ being locally analytic with compact support, we can find a finite set of disjoint charts $\{ X_i : i \in I \}$ of $N$ such that $f|_{X_i}$ is analytic for each $i$ and $f$ is $0$ outside $\bigcup X_i$. Applying Lemma~\ref{anESlemma} with $X = X_i$ we get $v_i \in V$ such that $f_{| X_i} = d_{-1} (v_i \tensor \mathbf{1}_{X_i})$. Hence $f = d_{-1} (\sum_{i \in I} v_i \tensor \mathbf{1}_{X_i})$ and we deduce that $\ker d_{0} \subseteq \im d_{-1}$.
\end{proof}

\section{Locally analytic principal series for subgroups with an Iwahori factorisation} \label{slocan}

In this section we complete the proof that \eqref{laES} is exact. To do this we have to introduce a particular kind of open compact subgroup of $G$.

\begin{definition}
We say an open compact subgroup $G_1 \subseteq G$ \textbf{admits an Iwahori factorisation} (with respect to $B$ and $\Bbar$) if multiplication induces an isomorphism of $L$-analytic manifolds
\[
(\Nbar_1) \times (T_1) \times (N_1) \overset{\sim}\longrightarrow G_1
\]
where $\Nbar_1 = \Nbar \cap G_1$, $T_1 = T \cap G_1$ and $N_1 = N \cap G_1$.
\end{definition}

The canonical example of an open compact subgroup of $G$ with an Iwahori factorisation is the Iwahori subgroup contained in a given special good maximal compact subgroup of $G$, and of type a given Borel subgroup. These are far from the only examples -- indeed Proposition~4.1.6 of \cite{emertonjacquet1} proves that we can find arbitrarily small such subgroups. Let us fix an open compact subgroup $G_1 \subseteq G$ which admits an Iwahori factorisation.

\begin{definition}
Let $\nu: T_1 \to K^{\times}$ be a locally analytic character. We extend it to a character $\nu: \Bbar_1 \to K^{\times}$ using $\Bbar \to \Bbar / \Nbar \cong T$. The \textbf{locally analytic principal series} associated to $G_1$ and $\nu$ is $\Ind_{\Bbar_1}^{G_1} (\nu)$.
\end{definition}

This has an action of $G_1$ by right translation. Since $(\Bbar_1) \backslash G_1 \cong N_1$ is compact, it follows from 4.1.5 of \cite{feauxdelacroix} that this is a locally analytic representation of $G_1$, and so we can differentiate the $G_1$-action to get an action of $\env(\g)$.

\begin{lemma} \label{admissible}
The locally analytic principal series $\Ind_{\Bbar_1}^{G_1} (\nu)$ is an admissible representation of $G_1$.
\end{lemma}

\begin{proof}
Proposition~6.4.iii of \cite{schnteit4} says that a closed $G_1$-invariant subspace of an admissible $G_1$-representation is an admissible $G_1$-representation. Since $\Ind_{\Bbar_1}^{G_1} (\nu)$ is isomorphic as a $G_1$-representation to a closed, $G_1$-invariant subspace of $\Ind_{\Bbar}^{G} (\nu)$, it is sufficient to show that $\Ind_{\Bbar}^{G} (\nu)$ is an admissible $G_1$-representation. By the definition of admissibility for locally analytic representations, this is equivalent to it being an admissible $G$-representation, which is shown at the beginning of \S6 of \cite{schnteit5}.
\end{proof}

In fact $\Ind_{\Bbar_1}^{G_1} (\nu)$ has an action of a monoid containing $G_1$. We define $T^{-} = \{ t \in T: t^{-1} N_1 t \subseteq N_1 \}$, which is a submonoid of $T$. We define $M$ to be $G_1 T^{-} G_1$, a submonoid of $G$ containing $G_1$. Then $\Bbar N_1 M \subseteq \Bbar N_1$, so the action of $M$ on $\Ind_{\Bbar}^{G} (\nu)$ preserves $\Ind_{\Bbar}^{G} (\nu)(N_1)$. We can identify $\Ind_{\Bbar_1}^{G_1} (\nu)$ with $\Ind_{\Bbar}^{G} (\nu) (N_1)$ using extension by $0$, which gives us an action of $M$ on $\Ind_{\Bbar_1}^{G_1}(\nu)$. We can also using the $\env(\g)$-equivariant isomorphism $\Ind_{\Bbar}^{G} (\nu) (N) \cong \mathcal{C}^{\la}_{\cs} (N, K_{\nu})$ to identify $\Ind_{\Bbar}^{G} (\nu)(N_1)$ with $\mathcal{C}^{\la}_{\cs} (N, K_{\nu})(N_1)$ and so transfer the $M$-action to $\mathcal{C}^{\la}_{\cs} (N, K_{\nu})(N_1)$.

\begin{lemma} \label{Mequivariant}
Given a non-zero morphism $\psi: \verma{\bbar}{\mu} \to \verma{\bbar}{\lambda}$, let $\phi: \mathcal{C}^{\la}_{\cs} (N, K_{\lambda})(N_1) \to \mathcal{C}^{\la}_{\cs} (N, K_{\mu})(N_1)$ denote the function obtained by restricting $\psi^{\la}$. Then $\phi$ is $M$-equivariant.
\end{lemma}

\begin{proof}
That $\phi$ is well-defined follows from Lemma~\ref{psiproperties}.\ref{parttwo}. Using the $M$-equivariant isomorphism $\mathcal{C}^{\la}_{\cs} (N, K_{\nu})(N_1) \to \Ind_{\Bbar}^{G} (\nu)(N_1)$ we can turn $\phi$ into a map $\Ind_{\Bbar}^{G} (\lambda)(N_1) \to \Ind_{\Bbar}^{G} (\mu)(N_1)$. This map is precisely $(u_{\psi})_L$, and the L action of $\g$ commutes with the right regular action of $M$.
\end{proof}

We define the smooth induction of the trivial character
\[
\smInd_{\Bbar}^{G} (\mathbf{1}) = \{ f \in \mathcal{C}^{\sm}(G, K): f(\overline{b} g) = f(g) \text{ for all } \overline{b} \in \Bbar, g \in G \}
\]
and we have $\mathcal{C}^{\sm}(N_1, K) \cong \smInd_{\Bbar_1}^{G_1} (\mathbf{1}) \cong \smInd_{\Bbar}^{G} (\mathbf{1})(N_1)$ as $\env(\g)$-modules. The same argument as for $\Ind_{\Bbar_1}^{G_1} (\nu)$ gives us an action of $M$ on all of these spaces.

\begin{proposition} \label{mainthm1}
When we restrict \eqref{laES} to functions with support in $N_1$ we get an exact sequence of $M$-representations
\begin{multline*}
0 \to V \tensor \mathcal{C}^{\sm}_{\cs} (N, K)(N_1) \xrightarrow{\delta_{-1}} \mathcal{C}^{\la}_{\cs} (N, K_{\lambda})(N_1) \xrightarrow{\delta_{0}} \bigoplus_{w \in W^{(1)}} \mathcal{C}^{\la}_{\cs} (N, K_{w \cdot \lambda})(N_1) \\
\cdots \xrightarrow{\delta_{i - 1}} \bigoplus_{w \in W^{(i)}} \mathcal{C}^{\la}_{\cs} (N, K_{w \cdot \lambda})(N_1) \xrightarrow{\delta_{i}} \cdots \xrightarrow{\delta_{r - 1}} \mathcal{C}^{\la}_{\cs} (N, K_{w_0 \cdot \lambda})(N_1) \to 0\end{multline*}
\end{proposition}

\begin{proof}
For $i \geq 0$, $\delta_i$ is well-defined by Lemma~\ref{psiproperties}.\ref{parttwo} and $M$-equivariant by Lemma~\ref{Mequivariant}. That $\delta_{-1}$ is well-defined and $M$-equivariant follows immediately from the definition of $d_{-1}$. (We use $d_i$ to refer to the maps in \eqref{laES}.)

By Theorem~\ref{laESthreetermthm} it is immediate that $\delta_{-1}$ is an injection and $\delta_{0} \circ \delta_{-1} = 0$. It also implies that $\ker \delta_{0} \subseteq \im \delta_{-1}$, by the following argument. Suppose that $f \in \mathcal{C}^{\la}_{\cs} (N, K_{\lambda})(N_1)$ is in $\ker \delta_{0}$. Then we know that it has a preimage $\sum v_i \tensor \mathbf{1}_{X_i} \in V \tensor_{K} \mathcal{C}^{\sm}_{\cs} (N, K)$ where the $X_i$ are disjoint charts of $N$. If we let $\phi$ denote the injection $V \to \mathcal{C}^{\pol} (N, K_{\lambda})$ from \eqref{polES} then $d_{-1} (\sum v_i \tensor \mathbf{1}_{X_i}) = \sum \phi(v_i) \mathbf{1}_{X_i}$, whence it follows that $X_i \subseteq N_1$ for all $i$ and $f \in \im \delta_{-1}$.

We now prove exactness at $\bigoplus_{w \in W^{(i)}} \mathcal{C}^{\la}_{\cs} (N, K_{w \cdot \lambda})(N_1)$ for $i \geq 1$. Since \eqref{lpES} is exact and $\mathcal{C}^{\la}_{\cs} (N, K_{\mu})$ is dense in $\mathcal{C}^{\la}_{\cs} (N, K_{\mu})$ we know that $d_{i} \circ d_{i - 1} = 0$, and hence $\delta_{i} \circ \delta_{i - 1} = 0$.  Thus, it suffices to prove that $\ker \delta_{i} \subseteq \im \delta_{i - 1}$.

Fix $(f_w)_{w \in W^{i}} \in \ker \delta_{i}$. Let us first suppose that we have a chart $X \subseteq N_1$ such that each $f_w$ is analytic on $X$ and $0$ outside it. Since $Z_{\lambda} \supseteq Z_{w \cdot \lambda}$ for all $w \in W$, by Lemma~\ref{sum}, we can write each $f_w$ uniquely as $\sum_{\nu \in Z_{\lambda}} f_{w, \nu}$ where $f_{w, \nu}|_{X} \in \mathcal{C}^{\pol} (X, w \cdot \lambda)_{\nu}$ and $f_{w, \nu}$ is $0$ outside $X$. Using the fact that $\delta_i$ is $\env(\g)$-equivariant, and applying Lemma~\ref{sum} with $\mu = w \cdot \lambda$ for each $w \in W^{(i + 1)}$, we see that $(f_{w, \nu})_{w \in W^{(i)}} \in \ker \delta_i$ for each $\nu \in Z_{\lambda}$. Since $Z_{\lambda}$ is countable let us choose an increasing sequence of finite subsets $A_n \subseteq Z_{\lambda}$ such that $\bigcup_{n=1}^{\infty} A_n = Z_{\lambda}$ and set $f_{w, n} = \sum_{\nu \in A_n} f_{w, \nu}$. Then $(f_{w, n})_{w \in W^{(i)}}$ tends to $(f_w)_{w \in W^{(i)}}$ as $n \to \infty$. By the exactness of \eqref{lpES}, each $(f_{w, n})_{w \in W^{(i)}}$ is in $\im d_{i - 1}$, and hence in $\im \delta_{i - 1}$ by Lemma~\ref{psiproperties}.\ref{partthree}. We want to show that their limit must therefore also be in $\im \delta_{i - 1}$. It is sufficient to demonstrate that $\im \delta_{i - 1}$ is closed.

As explained in Lemma~\ref{admissible}, $\mathcal{C}^{\la}_{\cs} (N, K_{\nu})(N_1) \cong  \Ind_{\Bbar_1}^{G_1} (\nu)$ is an admissible $G_1$-representation, and hence $\delta_{i - 1}$ is a $G_1$-equivariant, $K$-linear map between two admissible $G_1$-representations. By Proposition~6.4.ii in \cite{schnteit4}, the image of $\delta_{i - 1}$ is closed.

This deals with the case that there is a chart $X \subseteq N_1$ such that each $f_w$ is analytic on $X$ and $0$ outside of it. For a general $(f_w)_{w \in W^{(i)}} \in \ker \delta_i$ we can find a finite set of disjoint charts $\{ X_j \}$ which cover $N_1$ and such that for all $w \in W^{(i)}$ and all $j$, $f_w$ is analytic on $X_j$. We know that $((f_w)_{|X_j})_{w \in W^{(i)}}$ is still in $\ker \delta_{i}$, so by the above arguments we can find a preimage for it, and adding these all together we get a preimage for $(f_w)_{w \in W^{(i)}}$.
\end{proof}

\begin{corollary} \label{laEScor}
The sequence \eqref{laES} is exact.
\end{corollary}

\begin{proof}
Theorem~\ref{laESthreetermthm} deals with exactness at $V \tensor_{K} \mathcal{C}^{\sm}_{\cs} (N, K)$ and $\mathcal{C}^{\la}_{\cs} (N, K_{\lambda})$. Let $i \geq 1$. We explained in Proposition~\ref{mainthm1} that $d_i \circ d_{i - 1} = 0$, so it only remains to show that $\ker d_i \subseteq \im d_{i - 1}$.

Consider $(f_w)_{w \in W^{(i)}}$ in $\ker d_i$ such that for some $n \in N$, $\supp f_w \subseteq N_1 n$ for all $w \in W^{(i)}$. Then for all $w \in W^{(i)}$, $\supp n f_w \subseteq N_1$, and by Lemma~\ref{psiproperties}.\ref{partfive} we have $d_i ((n f_w)_{w \in W^{(i)}}) = n d_i ((f_w)_{w \in W^{(i)}}) = 0$. We proved in Proposition~\ref{mainthm1} that we therefore have a preimage $(g_w)_{w \in W^{(i - 1)}}$ of $(n f_w)_{w \in W^{(i)}}$. Then $((n^{-1} g_w)_{w \in W^{(i - 1)}}$ is a preimage of $(f_w)_{w \in W^{(i)}}$.

A general $(f_w)_{w \in W^{(i)}} \in \ker d_i$ can be written as a finite sum of such functions, so by linearity we are done.
\end{proof}

\begin{theorem} \label{laindEScor}
We have an exact sequence of $M$-representations
\begin{multline*}
0 \to V \tensor_{K} \smInd_{\Bbar_1}^{G_1} (\mathbf{1}) \to \Ind_{\Bbar_1}^{G_1} (\lambda) \to \bigoplus_{w \in W^{(1)}} \Ind_{\Bbar_1}^{G_1} (w \cdot \lambda) \to \\
\cdots \to \bigoplus_{w \in W^{(i)}} \Ind_{\Bbar_1}^{G_1} (w \cdot \lambda) \to \cdots \to \Ind_{\Bbar_1}^{G_1} (w_0 \cdot \lambda) \to 0.
\end{multline*}
coming from the BGG resolution for $V^{*}$.
\end{theorem}

\begin{proof}
This follows immediately from Proposition~\ref{mainthm1}.
\end{proof}

\section{Locally analytic principal series for $G$} \label{sG}

Using the results of \S\ref{slocan} we now construct an exact series analogous that in Theorem~\ref{laindEScor} for locally analytic principal series for $G$. Let $G_0$ be a special good maximal compact subgroup of $G$. The Iwasawa decomposition says that $G = \Bbar G_0$ (cf. \S3.5 in \cite{Cartier}), which gives us an isomorphism of $G_0$-representations
\[
\Ind_{\Bbar}^{G} (\nu) \cong \Ind_{\Bbar_0}^{G_0} (\nu)
\]
where $\Bbar_0 = \Bbar \cap G_0$.

We may fix representatives of $W$ which are in $G_0$, by 4.2.3 of \cite{BruhatTits}.

Let $G_1 \subseteq G_0$ be the Iwahori subgroup of the same type as $\Bbar$. This has an Iwahori factorisation with respect to $B$ and $\Bbar$ and we have the Bruhat-Iwahori decomposition
\[
G_0 = \bigsqcup_{w \in W} \Bbar_0 w G_1.
\]
It follows that we have an isomorphism of $G_1$-representations
\begin{align*}
\Ind_{\Bbar_0}^{G_0} (\nu) \longrightarrow \bigoplus_{w \in W} \Ind_{\Bbar \cap w G_1 w^{-1}}^{w G_1 w^{-1}} (\nu) && f \longmapsto ((w f)|_{w G_1 w^{-1}})_{w \in W}
\end{align*}
where the action of $G_1$ on $\Ind_{\Bbar \cap w G_1 w^{-1}}^{w G_1 w^{-1}} (\nu)$ is via $G_1 \to w G_1 w^{-1}$, $g \mapsto w g w^{-1}$.

\begin{lemma} \label{iwahorifac}
For any $w \in W$, $w G_1 w^{-1}$ has an Iwahori factorisation
\[
(w G_1 w^{-1} \cap \Nbar) \times (w G_1 w^{-1} \cap T) \times (w G_1 w^{-1} \cap N) \overset{\sim}\longrightarrow w G_1 w^{-1}
\]
with respect to $B$ and $\Bbar$.
\end{lemma}

\begin{proof}
This follows from Lemme~5.4.2 in \cite{matsumoto}.
\end{proof}

In \S\ref{sdefmaps} we started with a map $\psi: \verma{\bbar}{\mu} \to \verma{\bbar}{\lambda}$ and produced $\psi^{\la}: \mathcal{C}^{\la}_{\cs} (N, K_{\lambda}) \to \mathcal{C}^{\la}_{\cs} (N, K_{\mu})$. This leads to a $G_1$-equivariant map
\[
(u_{\psi})_{L}: \Ind_{\Bbar \cap w G_1 w^{-1}}^{w G_1 w^{-1}} (\lambda) \to \Ind_{\Bbar \cap w G_1 w^{-1}}^{w G_1 w^{-1}} (\mu)
\]

\begin{lemma} \label{psiGequiv}
Using $\Ind_{\Bbar}^{G}(\nu) \cong \Ind_{\Bbar_0}^{G_0} (\nu) \cong \bigoplus_{w \in W} \Ind_{\Bbar \cap w G_1 w^{-1}}^{w G_1 w^{-1}} (\nu)$, the above maps give us a $G_1$-equivariant map $\Ind_{\Bbar}^{G} (\lambda) \to \Ind_{\Bbar}^{G} (\mu)$. It is in fact $G$-equivariant.
\end{lemma}

\begin{proof}
We will show that this map is precisely $(u_{\psi})_{L}$, which is $G$-equivariant. Let $f \in \Ind_{\Bbar}^{G} (\lambda)$. This corresponds to
\[
((wf)|_{w G_1 w^{-1}})_{w \in W} \in \bigoplus_{w \in W} \Ind_{\Bbar \cap w G_1 w^{-1}}^{w G_1 w^{-1}} (\lambda)
\]
which is in turn sent to
\[((u_{\psi})_{L} (wf)|_{w G_1 w^{-1}})_{w \in W} \in \bigoplus_{w \in W} \Ind_{\Bbar \cap w G_1 w^{-1}}^{w G_1 w^{-1}} (\mu).
\]
There is a unique $f' \in \Ind_{\Bbar}^{G} (\mu)$ such that $(w f')|_{w G_1 w^{-1}} = (u_{\psi})_{L} (wf)|_{w G_1 w^{-1}}$ for all $w \in W$. Since
\[
(u_{\psi})_{L} (wf)|_{w G_1 w^{-1}} = ((u_{\psi})_{L} wf)|_{w G_1 w^{-1}} = (w (u_{\psi})_{L} f)|_{w G_1 w^{-1}}
\]
the obvious candidate for $f'$ is $(u_{\psi})_{L} f$. We must show that $(u_{\psi})_{L} f \in \Ind_{\Bbar}^{G} (\mu)$. Since $G = \Bbar G_0 = \Bbar \left( \bigsqcup_{w \in W} \Bbar_0 w G_1 \right)  = \bigsqcup_{w \in W} \Bbar w G_1$, it suffices to prove that $(u_{\psi})_{L} f(\overline{b} w g) = \mu (\overline{b}) (u_{\psi})_{L} f(w g)$ for all $\overline{b} \in \Bbar$, $w \in W$ and $g \in G_{1}$.

Fix $w \in W$. We have $w f \in \Ind_{\Bbar}^{G} (\lambda)$ and hence $(wf)_{| N} \in \Ind_{\Bbar}^{G} (\lambda) (N)$. In the proof of Theorem~\ref{psiextends} we showed that $(u_{\psi})_{L} \Ind_{\Bbar}^{G} (\lambda) (N) \subseteq \Ind_{\Bbar}^{G} (\mu)(N)$, so $(u_{\psi})_{L} ((w f)_{| N}) \in \Ind_{\Bbar}^{G} (\mu)(N)$. Let $\overline{b} \in \Bbar$ and $g \in G_1$. By Lemma~\ref{uLrestrict}, $(u_{\psi})_{L} ((w f)_{| N}) = ((u_{\psi})_{L} w f)_{| N})$, and $w g w^{-1} \in \Bbar N$ by Lemma~\ref{iwahorifac}, so
\begin{align*}
((u_{\psi})_{L} w f)) (\overline{b} w g w^{-1}) & = ((u_{\psi})_{L} w f)_{| N}) (\overline{b} w g w^{-1}) \\
& = \mu(\overline{b}) ((u_{\psi})_{L} w f)_{| N}) (w g w^{-1}) \\
& = \mu(\overline{b}) ((u_{\psi})_{L} w f)) (w g w^{-1})
\end{align*}
whence it immediately follows that $(u_{\psi})_{L} f(\overline{b} w g) = \mu (\overline{b}) (u_{\psi})_{L} f(w g)$.
\end{proof}

We can now prove an analogue of Theorem~\ref{laindEScor} for locally analytic principal series for all of $G$. This has been done independently by different methods in \S4.9 of \cite{OrlikStrauch2}.

\begin{theorem}
We have an exact sequence of $G$-representations
\begin{multline*}
0 \to V \tensor \smInd_{\Bbar}^{G} (\mathbf{1}) \to \Ind_{\Bbar}^{G} (\lambda) \to \bigoplus_{w \in W^{(1)}} \Ind_{\Bbar}^{G} (w \cdot \lambda) \to \\
\cdots \to \bigoplus_{w \in W^{(i)}} \Ind_{\Bbar}^{G} (w \cdot \lambda) \to \cdots \to \Ind_{\Bbar}^{G} (w_0 \cdot \lambda) \to 0
\end{multline*}
coming from the BGG resolution for $V^{*}$.
\end{theorem}

\begin{proof}
For each $w \in W$ we have an exact sequence of $w G_1 w^{-1}$-representations
\begin{multline*}
0 \to V \tensor \smInd_{\Bbar \cap w G_1 w^{-1}}^{w G_1 w^{-1}} (\mathbf{1}) \to \Ind_{\Bbar \cap w G_1 w^{-1}}^{w G_1 w^{-1}} (\lambda) \to \bigoplus_{w \in W^{(1)}} \Ind_{\Bbar \cap w G_1 w^{-1}}^{w G_1 w^{-1}} (w \cdot \lambda) \\
\cdots \to \bigoplus_{w \in W^{(i)}} \Ind_{\Bbar \cap w G_1 w^{-1}}^{w G_1 w^{-1}} (w \cdot \lambda) \to \cdots \to \Ind_{\Bbar \cap w G_1 w^{-1}}^{w G_1 w^{-1}} (w_0 \cdot \lambda) \to 0
\end{multline*}
by Lemma~\ref{iwahorifac} and the results of \S\ref{slocan}. Letting $G_1$ act via $G_1 \to w G_1 w^{-1}$, taking the direct sum of these exact sequences over all $w \in W$ and using the isomorphism of $G_1$-representations $\Ind_{\Bbar}^{G} (\nu) \cong \bigoplus_{w \in W} \Ind_{\Bbar \cap w G_1 w^{-1}}^{w G_1 w^{-1}} (\nu)$ and its smooth analogue we get the required exact sequence, but only as an exact sequence of $G_1$-representations. It remains to show that the maps are $G$-equivariant.

First consider $d_{-1}: V \tensor \smInd_{\Bbar}^{G} (\mathbf{1}) \to \Ind_{\Bbar}^{G} (\lambda)$. Given $v \tensor f \in V \tensor \smInd_{\Bbar}^{G} (\mathbf{1})$ we construct $d_{-1} (v \tensor f)$ as follows. First we send $v \tensor f$ to
\[
(wv \tensor wf|_{w G_1 w^{-1}})_{w \in W} \in \bigoplus_{w \in W} V \tensor \smInd_{\Bbar \cap w G_1 w^{-1}}^{w G_1 w^{-1}} (\mathbf{1}).
\]
We then apply the maps
\begin{align*}
V \tensor \smInd_{\Bbar \cap w G_1 w^{-1}}^{w G_1 w^{-1}} (\mathbf{1}) \to \Ind_{\Bbar \cap w G_1 w^{-1}}^{w G_1 w^{-1}} (\lambda) && v \tensor f \mapsto \phi(v)|_{w G_1 w^{-1}} f
\end{align*}
where $\phi$ is the $G$-equivariant isomorphism from $V$ to the algebraic induction of $\lambda$ from $\Bbar$ to $G$. This gives us $(\phi(wv)(wf)|_{w G_1 w^{-1}})_{w \in W}$, which can be expressed as $(w (\phi(v) f)|_{w G_1 w^{-1}})_{w \in W}$. Applying the inverse of the isomorphism $\Ind_{\Bbar_0}^{G_0} (\lambda) \cong \bigoplus_{w \in W} \Ind_{\Bbar \cap w G_1 w^{-1}}^{w G_1 w^{-1}} (\lambda)$ we get that $d_{-1} (v \tensor f) = \phi(v) f$ and hence $d_{-1}$ is $G$-equivariant.

The $G$-equivariance of $d_{i}: \bigoplus_{w \in W^{(i)}} \Ind_{\Bbar}^{G} (w \cdot \lambda) \to \bigoplus_{w \in W^{(i + 1)}} \Ind_{\Bbar}^{G} (w \cdot \lambda)$ for $i \geq 0$ follows easily from Lemma~\ref{psiGequiv} and the fact that the maps $\bigoplus_{w \in W^{(i)}} \Ind_{\Bbar \cap w G_1 w^{-1}}^{w G_1 w^{-1}} (w \cdot \lambda) \to \bigoplus_{w \in W^{(i)}} \Ind_{\Bbar \cap w G_1 w^{-1}}^{w G_1 w^{-1}} (w \cdot \lambda)$ are constructed from maps of the form $(u_{\theta_{\alpha, w \cdot \lambda}})_{L}$.
\end{proof}

\section{Analytic principal series for subgroups with an Iwahori factorisation} \label{san}

Let $G_1$ be an open compact subgroup of $G$ which admits an Iwahori factorisation, and such that there is a locally analytic isomorphism between $N_1$ and the $L$-points of a rigid analytic space which is compatible with all charts of $N$. Let $\nu: T_1 \to K^{\times}$ be a locally analytic character. We extend it to a character $\nu: \Bbar_1 \to K^{\times}$ using $\Bbar \to \Bbar / \Nbar \cong T$.

\begin{definition}
The \textbf{analytic principal series} associated to $G_1$ and $\nu$ is
\[
\anInd_{\Bbar_1}^{G_1} (\nu) = \{ f \in \Ind_{\Bbar_1}^{G_1} (\nu): f \text{ is analytic on } N_1 \}.
\]
\end{definition}

From now on we assume that $\nu$ comes from restricting an element of $X(\mathbf{T})$, as this is the case we are interested in.

The action of $\env(\g)$ on $\Ind_{\Bbar_1}^{G_1} (\nu)$ preserves $\anInd_{\Bbar_1}^{G_1} (\nu)$ because the right regular action of $\g$ on $\mathcal{C}^{\la}(G, K)$ is via differential operators, which preserve the property of being analytic on $N_1$. We use this to give $\anInd_{\Bbar_1}^{G_1} (\nu)$ an action of $\env(\g)$.

\begin{lemma} \label{ManInd}
The action of $M$ on $\Ind_{\Bbar_1}^{G_1} (\nu)$ preserves $\anInd_{\Bbar_1}^{G_1} (\nu)$.
\end{lemma}

\begin{proof}
Consider the image of $\anInd_{\Bbar_1}^{G_1} (\nu)$ under the isomorphism $\Ind_{\Bbar_1}^{G_1} (\nu) \cong \Ind_{\Bbar}^{G} (\nu) (N_1)$. Since $\nu$ is analytic it consists of all functions which are analytic on $\Bbar G_1 = \Bbar N_1$ and $0$ outside it. Since $\Bbar G_1 M \subseteq \Bbar G_1$, this is preserved by the action of $M$.
\end{proof}

We use this to give $\anInd_{\Bbar_1}^{G_1} (\nu)$ an action of $M$.

\begin{theorem} \label{anmainthm}
The sequence \eqref{laES} gives an exact sequence of $M$-representations
\[
0 \longrightarrow V \longrightarrow \anInd_{\Bbar_1}^{G_1} (\lambda) \longrightarrow \bigoplus_{w \in W^{(1)}} \anInd_{\Bbar_1}^{G_1} (w \cdot \lambda).
\]
\end{theorem}

\begin{proof}
Setting $X = N_1$ and using the isomorphism $\anInd_{\Bbar_1}^{G_1} (\nu) \cong \mathcal{C}^{\an} (N_1, K_{\nu})$, we showed this sequence was exact in  Lemma~\ref{anESlemma}. The maps are $M$-equivariant because they are the restriction of maps from the exact sequence in Theorem~\ref{laindEScor}, which are $M$-equivariant, and we have shown the spaces are $M$-stable.
\end{proof}

The analogue of the whole of \eqref{laES} with analytic principal series is a chain complex but is not in general exact.

\section{Applications to overconvergent $p$-adic automorphic forms I} \label{sappi}

In this section we outline the definition of spaces of overconvergent $p$-adic automorphic forms given in \cite{Chenevierfern} and construct an exact sequence between certain such spaces. This has already been done in \cite{Chenevierfern} but is included here for completeness.

Let $F$ be a number field. Let $\mathbf{U}$ be an algebraic group defined over $F$ such that $\mathbf{U}(F_v)$ is compact for all infinite places $v$ of $F$ and $\mathbf{U}(F_v) \cong GL_n(\Q_p)$ for all places $v$ of $F$ dividing $p$. Let $S_p$ denote the set of all places of $F$ dividing $p$ and fix an isomorphism $\mathbf{U}(F_v) \cong GL_n(\Q_p)$ for all $v \in S_p$.

Let $\mathbf{G}$ be the algebraic group $\mathbf{GL}_n^{S_p}$ defined over $\Q_p$. Let $G = \mathbf{G}(\Q_p)$, $B \subseteq G$ the Borel consisting of lower triangular matrices and $T \subseteq G$ the maximal torus consisting of diagonal matrices. Define $G_1 \subseteq G$ to be the Iwahori subgroup of $GL_n(\Z_p)$ coming from $\Bbar$. (Because of differing conventions what we call $\mathbf{B}$ is called $\Bbar$ in \cite{Chenevierfern} and vice versa.)

Let $\mathbb{A}_f$ denote the finite ad\`eles and $\mathbb{A}_f^{S_p}$ the finite ad\`eles away from $v \in S_p$. Fix an open compact subgroup $\mathscr{U}$ of $\mathbf{U}( \mathbb{A}_f )$ of the form $G_1 \times \mathscr{U}^{S_p}$ where $\mathscr{U}^{S_p}$ is an open compact subgroup of $\mathbf{U}(\mathbb{A}_f^{S_p})$. Consider the functor $F$ from representations of $M = G_1 T^{-} G_1$ over $\Q_p$ to $\Q_p$-vector spaces given by setting $F(A)$ to be the set of all functions $\phi: \mathbf{U}(F) \backslash \mathbf{U}(\mathbb{A}_{F, f}) \to A$ such that $\phi (g x) = (\prod_{v | p} x_v)^{-1} \phi(g)$ for all $g \in \mathbf{U}(\mathbb{A}_{F, f})$ and $x \in \mathscr{U}$. This is an exact functor.

For $\chi \in X(\mathbf{T})$ Chenevier defines a representation $\mathcal{C}_{\chi}$ of $M$ which can easily be shown to be isomorphic to $\anInd_{\Bbar_1}^{G_1} (-\chi)$. (Recall the group operation on $X(\mathbf{T})$ is written additively, so $(-\chi)(t) = \chi(t)^{-1}$.) He defines the space of automorphic forms of $\mathbf{U}$ of weight $\chi$ and level $\mathscr{U}$ to be $F( \mathcal{C}_{\chi})$.

\begin{theorem} \label{ChenevierES}
Let $V$ be a finite dimensional irreducible algebraic representation of $\mathbf{G}$, with lowest weight $\lambda \in X(\mathbf{T})$. We have an exact sequence
\[
0 \to F(V^{*}) \to F( \mathcal{C}_{\lambda}) \to \bigoplus_{w \in W^{(1)}} F( \mathcal{C}_{w \cdot \lambda})
\]
\end{theorem}

\begin{proof}
Consider the exact sequence in Theorem~\ref{anmainthm} for $V^{*}$, which has highest weight $-\lambda$. Applying the functor $F$ we get the required exact sequence.
\end{proof}

Note that when we talk about highest and lowest weights we mean with respect to the choice of positive roots given by $\mathbf{B}$. Since Chenevier takes our $\mathbf{\Bbar}$ for his choice of positive roots, in his terminology $V$ has highest weight $\lambda$.

Chenevier calls $F(V^{*})$ the space of classical overconvergent $p$-adic automorphic forms.

\section{Applications to overconvergent $p$-adic automorphic forms II} \label{sappii}

In this section we outline the definition of spaces of overconvergent $p$-adic automorphic forms given in \cite{Loeffler} and construct an exact sequence involving them.

Choose a number field $F$ and a prime $\mathfrak{p}$ of $F$. Let $\mathbf{H}$ be a connected reductive algebraic group defined over $F$ such that $\mathbf{H}(F \tensor_{\Q} \R)$ is compact modulo centre. Write $H_{\infty}^{0}$ for the identity component of $\mathbf{H}(F \tensor_{\Q} \R)$. Let $\mathbb{A}_{F}$ denote the ad\`eles of $F$, $\mathbb{A}_{F, f}$ the finite ad\`eles of $F$ and $\mathbb{A}_{F, f}^{(\mathfrak{p})}$ the finite ad\`eles of $F$ away from $\mathfrak{p}$. Let $L = F_{\mathfrak{p}}$ and let $\mathbf{G}$ be the base change of $\mathbf{H}$ to $L$. Assume that $\mathbf{G}$ is quasi-split.

We are now in the situation of \cite{Loeffler}, with the added assumption that the parabolic subgroup $P \subseteq \mathbf{H}(F_{\mathfrak{p}})$ is a Borel. Let us now outline the definition of the space of overconvergent $p$-adic automorphic forms for $\mathbf{H}$ used in \cite{Loeffler}. In the terminology of \cite{Loeffler}, we consider only the case where $X$ in arithmetic weight space is in fact the singleton $\mathbf{1}$ consisting of the trivial weight and $V$ is a one-dimensional representation of $T_1$ of the form $K_{\mu}$ for $\mu \in X(\mathbf{T})$ which is an arithmetical character. The field called $E$ in \cite{Loeffler} we call $K$, the group called $G_0$ we call $G_1$ and the monoid called $\mathbb{I}$ we call $M$. We put the extra condition on $G_1$ that if $t \in T$ such that $| \alpha (t)| < 1$ for all $\alpha \in \Delta$ then $t N_1 t^{-1} \subseteq N_1$ and $t^{-1} \Nbar_1 t \subseteq \Nbar_1$.

A representation of $M = G_1 T^{-} G_1$ over $K$ is said to be arithmetical if there is a finite index subgroup in $\mathbf{Z_H}(\mathfrak{o}_F) \subseteq M$ which acts trivially.

For an arithmetical representation $A$ of $M$ over $K$ we define $\mathcal{L} (A)$ to be the set of all functions $\phi: \mathbf{H}(F) \backslash \mathbf{H}(\mathbb{A}_{F}) \to A$ such that there exists some open subset $U \subseteq \mathbf{H}(\mathbb{A}_{F, f}^{(\mathfrak{p})}) \times G_1$ (which can depend on $\phi$) with $\phi(h u) = u_{\mathfrak{p}}^{-1} \phi(h)$ for all $u \in U \times H_{\infty}^{0}$ and $h \in \mathbf{H}(\mathbb{A}_{F})$.

For a sufficiently large integer $k$ the space of $k$-overconvergent $p$-adic automorphic forms $M(e, \mathbf{1}, V, k)$  for $\mathbf{H}$ with weight $( \mathbf{1}, K_{\mu})$ and type $e$ is defined to be $e \mathcal{L} (\mathcal{C}( \mathbf{1}, K_{\mu}, k))$. Here $e$ is an idempotent in a certain Hecke algebra $\mathcal{H}^{+}(\mathcal{G})$ which corresponds to the tame level -- see \cite{Loeffler} for more details, and for the definition of $\mathcal{C}( \mathbf{1}, K_{\mu}, k)$.

For $k$ large enough that $\mathcal{C}( \mathbf{1}, K_{\mu}, k)$ is defined there is a natural map $\mathcal{C}( \mathbf{1}, K_{\mu}, k) \to \mathcal{C}( \mathbf{1}, K_{\mu}, k + 1)$, so by functoriality there is a natural map $e \mathcal{L} (\mathcal{C}( \mathbf{1}, K_{\mu}, k)) \to e \mathcal{L} (\mathcal{C}( \mathbf{1}, K_{\mu}, k + 1))$ (which is injective with dense image). We make the following definition.

\begin{definition}
The \textbf{space $M(e, K_{\mu})$ of overconvergent $p$-adic automorphic forms} of weight $K_{\mu}$ and type $e$ is defined to be $\varinjlim_k M(e, \mathbf{1}, K_{\mu}, k))$.
\end{definition}

In the proof of Proposition~3.10.1 in \cite{Loeffler} we see that $\varinjlim_k M(e, \mathbf{1}, K_{\mu}, k))$ is isomorphic to $e \mathcal{L} (\Ind_{\Bbar_1}^{G_1} (\mu))$, so we have $M(e, K_{\mu}) = e \mathcal{L} (\Ind_{\Bbar_1}^{G_1} (\mu))$.

We define the classical subspace $M(e, K_{\mu})^{\cl}$ to be $e \mathcal{L} (\Ind_{\Bbar_1}^{G_1} (\mu)^{\cl})$, where $\Ind_{\Bbar_1}^{G_1} (\mu)^{\cl}$ is the intersection of $\Ind_{\Bbar_1}^{G_1} (\mu)$ with the image of $\mathcal{C}^{\pol} (G_1, K) \tensor_{K} \mathcal{C}^{\sm} (G_1, K)$ under the natural multiplication map, i.e. the subrepresentation of locally polynomial vectors in $\Ind_{\Bbar_1}^{G_1} (\mu)$. In particular, $\Ind_{\Bbar_1}^{G_1} (\lambda)^{\cl} = V \tensor_{K} \smInd_{\Bbar_1}^{G_1} (\mathbf{1})$.

\begin{theorem} \label{LoefflerES}
If $\lambda \in X(\mathbf{T})$ is a dominant arithmetical character then we have an exact sequence
\begin{multline*}
0 \to M(e, K_{\lambda})^{\cl} \to M(e, K_{\lambda}) \to \bigoplus_{w \in W^{(1)}} M(e, K_{w \cdot \lambda}) \to \\
\cdots \to \bigoplus_{w \in W^{(i)}} M(e, K_{w \cdot \lambda}) \to \cdots \to M(e, K_{w_0 \cdot \lambda}) \to 0
\end{multline*}
\end{theorem}

\begin{proof}
We first show that all the terms in the exact sequence in Theorem~\ref{laindEScor} are arithmetical. In the proof of Theorem~\ref{laESred} we showed that $w \cdot \lambda|_{\mathbf{Z_G}} = \lambda|_{\mathbf{Z_G}}$ for all $w \in W$. As $K_{\lambda}$ is arithmetical and $\mathbf{Z_H}(\mathfrak{o}_F) \subseteq \mathbf{Z_G}(L)$, we see that $K_{w \cdot \lambda}$ is arithmetical for all $w \in W$. Since $\mathbf{Z_H}(\mathfrak{o}_F)$ acts on $\Ind_{\Bbar_1}^{G_1} (\mu)$ via the same character that it acts on $K_{\mu}$, it follows that $\Ind_{\Bbar_1}^{G_1} (w \cdot \lambda)$ is arithmetical for all $w \in W$. Finally, $V \tensor_{K} \smInd_{\Bbar_1}^{G_1} (\mathbf{1})$ injects into an arithmetical representation and is therefore also arithmetical.

As explained in the proof of Corollary~3.3.5 in \cite{Loeffler}, the functor $e \mathcal{L}$ on the category of arithmetic representations is the same as taking the image of an idempotent in a finite-dimensional matrix algebra over the group ring $K[G_1]$. It is hence exact, and applying it to the exact sequence in Theorem~\ref{laindEScor} we get the required exact sequence.
\end{proof}

The maps in the exact sequence in Theorem~\ref{laindEScor} are made up of maps of the form $(u_{\psi})_L : \Ind_{\Bbar_1}^{G_1} (w \cdot \lambda) \to \Ind_{\Bbar_1}^{G_1} (s_{\alpha} w \cdot \lambda)$. Given one of these maps, we define $\theta_{\alpha, w \cdot \lambda}^{\text{aut}}$ to be $e \mathcal{L} ((u_{\psi})_L) : M(e, K_{w \cdot \lambda}) \to M(e, K_{s_{\alpha} w \cdot \lambda})$. All the maps in the exact sequence in Theorem~\ref{LoefflerES} after the first are made up from these $\theta_{\alpha, w \cdot \lambda}^{\text{aut}}$. In particular, $M(e, K_{\lambda}) \to \bigoplus_{w \in W^{(1)}} M(e, K_{w \cdot \lambda})$ is $\bigoplus_{\alpha \in \Delta} \theta_{\alpha, \lambda}^{\text{aut}}$, from which we deduce that for any dominant weight $\lambda \in X(\mathbf{T})$, $f \in M(e, K_{\lambda})$ is in $M(e, K_{\lambda})^{\cl}$ if and only if $f \in \ker \theta_{\alpha, \lambda}^{\text{aut}}$ for all $\alpha \in \Delta$.

\section*{Acknowledgements}

I would like to thank Ga\"{e}tan Chenevier, David Loeffler, James Newton and my supervisor Kevin Buzzard.



%% References




\comment{
\bibliographystyle{plain}
\bibliography{../bib}
}



% \end{thebibliography}

\end{document}